\documentclass{amsart}
\usepackage{amsfonts,amssymb,amscd,amsmath,color,enumerate,verbatim}
\usepackage[all]{xy}
\xyoption{curve}
\SelectTips{eu}{}
\CompileMatrices

\usepackage[colorlinks]{hyperref}

\usepackage{calc}

\CompileMatrices

\makeatletter
\@namedef{subjclassname@2020}{\textup{2020} Mathematics Subject Classification}
\makeatother

\def\col{{\colon}}

\def\ev{{\mathsf{even}}}
\def\od{{\mathsf{odd}}}

\def\ov{\overline}

\def\wJ{\widetilde{J}}
\def\wh{\widehat}

\def\xra{\xrightarrow}
\def\tla{\twoheadleftarrow}
\def\hra{\hookrightarrow}
\def\tra{\twoheadrightarrow}

\def\lin#1#2{{#1}_{\langle #2\rangle}}

\def\fm{{\mathfrak m}}

\def\fn{{\mathfrak n}}
\def\fp{{\mathfrak p}}

\def\fq{{\mathfrak q}}

\def\JJ{{\mathbb J}}

\def\NN{{\mathbb N}}
\def\QQ{{\mathbb Q}}
\def\RR{{\mathbb R}}

\def\ZZ{{\mathbb Z}}

\def\bsa{{\boldsymbol a}}

\def\bsf{{\boldsymbol f}}
\def\bsg{{\boldsymbol g}}

\def\bst{{\boldsymbol t}}
\def\bsu{{\boldsymbol u}}

\def\bsx{{\boldsymbol x}}
\def\bsy{{\boldsymbol y}}
\def\bsz{{\boldsymbol z}}

\def\CK{{\mathcal K}}

\def\CM{{\mathcal M}}
\def\CN{{\mathcal N}}
\def\CP{{\mathcal P}}

\def\CR{{\mathcal R}}

\def\CU{{\mathcal U}}
\def\CV{{\mathcal V}}

\def\CX{{\mathcal X}}

\def\les{{\leqslant}}
\def\ges{{\geqslant}}

\def\gr#1{#1{}^{\mathsf g}}

\def\sq{\scriptscriptstyle\square}
\def\gri#1#2{{#1^{\mathsf g}_{#2}}}

\def\qr#1{{{#1}{}^{\scriptscriptstyle\square}}}
\def\qri#1#2{{{#1}_{#2}^{\sq}}}

\def\Ass{\operatorname{Ass}}

\def\Ann{\operatorname{Ann}}

\def\Tor{\operatorname{Tor}}

\def\Ext{\operatorname{Ext}}

\def\Hom{\operatorname{Hom}}
\def\Ker{\operatorname{Ker}}

\def\crdeg{\operatorname{cr\,deg}}

\def\rank{\operatorname{rank}}

\def\deg{\operatorname{deg}}

\def\codepth{\operatorname{codepth}}
\def\codim{\operatorname{codim}}
\def\gran{\operatorname{gn}}
\def\cx{\operatorname{cx}}

\def\depth{\operatorname{depth}}
\def\grade{\operatorname{grade}}
\def\edim{\operatorname{edim}}

\def\height{\operatorname{height}}

\def\rel{\operatorname{rel}}

\def\pd{\operatorname{proj\,dim}}

\def\Sym{\operatorname{Sym}}

\swapnumbers

\theoremstyle{plain}
\newtheorem{Theorem}{Theorem}[section]
\newtheorem{Proposition}[Theorem]{Proposition}
\newtheorem{Lemma}[Theorem]{Lemma}
\newtheorem{Corollary}[Theorem]{Corollary}

\theoremstyle{definition}

\newtheorem{Question}[Theorem]{Question}

\newtheorem{chunk}[Theorem]{}

\theoremstyle{remark}

\newtheorem{remark}[Theorem]{Remark}

\numberwithin{equation}{Theorem}
	
\begin{document}

\title[Polynomial growth of Betti sequences]
{Polynomial growth of Betti sequences \\ over local rings}

\author[L.~L.~Avramov]{Luchezar L.~Avramov}
\address{Luchezar L.~Avramov\\ Department of Mathematics\\
   University of Nebraska\\ Lincoln\\ NE 68588\\ U.S.A.}
     \email{avramov@unl.edu}

\author[A.~Seceleanu]{Alexandra~Seceleanu}
\address{Alexandra.~Seceleanu\\ Department of Mathematics\\
   University of Nebraska\\ Lincoln\\ NE 68588\\ U.S.A.}
     \email{aseceleanu@unl.edu}

\author[Z.~Yang]{Zheng Yang}
\address{Zheng Yang\\
   Sichuan University -- Pittsburgh Institute\\ Chengdu\\  Sichuan\\ China.}
 \email{zhengyang2018@scu.edu.cn}

  \begin{abstract}
This is a study of the sequences of Betti numbers of finitely generated modules over a complete intersection local ring, $R$.  The subsequences $(\beta^R_i(M))$ with even, respectively, odd $i$ are known to be eventually given by polynomials in $i$ with equal leading terms.  We show that these polynomials coincide if $\qr I$, the ideal generated by the quadratic relations of the associated graded ring of $R$, satisfies $\height \qr I \ge \codim R -1$, and that the converse holds if $R$ is homogeneous or $\codim R \le 4$.  Subsequently Avramov, Packauskas, and Walker proved that the terms of degree $j > \codim R -\height \qr I$ of the even and odd Betti polynomials are equal.  We give a new proof of that result, based on an intrinsic characterization of residue rings of c.i.\ local rings of minimal multiplicity obtained in this paper.  We also show that that bound is optimal.  
   \end{abstract}


\thanks{Research partly supported by National Science Foundation grants DMS-1103176, DMS--210122 and National Science Foundation of China grant 12001384.}

  \maketitle

\tableofcontents

\section{Introduction}
   \label{S:Intro}

This paper is concerned with free resolutions of finitely generated modules $M$ over a 
commutative noetherian ring $R$ with unique maximal ideal, $\fm$.  Each such module
has a unique up to isomorphism \emph{minimal} free resolution.  The rank $\beta^R_i(M)$ 
of the $i$th module in such a resolution is called the $i$th \emph{Betti number} of~$M$.

The asymptotic patterns of \emph{Betti sequences} $(\beta^R_i(M))$ reflect and affect the 
singularity of~$R$.  This dynamic is best understood when the ring $R$ is \emph{complete 
intersection}, abbreviated to \emph{c.i.}; that is, when the $\fm$-adic completion $\wh R$ is 
isomorphic to the residue ring of some regular local ring modulo an ideal generated by a regular 
set; the smallest cardinality of such a set is equal to $\codim R$, the \emph{codimension} of~$R$.  

Gulliksen \cite{Gu1} proved that if $R$ is c.i., then for every $M$ there exist \emph{Betti polynomials}, 
$\beta^{R,M}_0$ and $\beta^{R,M}_1 \in \QQ[x]$ with $\deg(\beta^{R,M}_j) < \codim R$ (where 
$\deg(0) := -1$, by convention) such that $\beta^R_{i}(M) = \beta^{R,M}_j(i)$ for $i \gg 0$ and 
$i \equiv j \pmod 2$.  The hypothesis on $R$ cannot be relaxed, as $\beta^R_i(k) \le b(i)$ with 
$k := R/\fm$ and $b \in \RR[x]$ implies $R$ is c.i.\ (Gulliksen, \cite{Gu2}), nor can the conclusion 
on $\beta^{R,M}_j$ be tightened, for $\beta^{R,k}_{\ev} = \beta^{R,k}_{\od}$ and 
$\deg(\beta^{R,k}_{\ev}) = \codim R -1$ hold when $R$ is c.i.\ (Tate, \cite{Ta}).

Eisenbud \cite{Ei1} showed that if $R$ is c.i.\ and $\codim R \le 1$, then $(\beta^R_i(M))$ is eventually 
constant for every $M$; this was an early sign of possible connections between $\beta^{R,M}_{\ev}$ and 
$\beta^{R,M}_{\od}$.  The general property is that these polynomials have equal degrees and leading 
coefficients over every c.i.\ ring; see Avramov~\cite{Av:vpd}.  The present work is a study of the 
discrepancy between $\beta^{R,M}_{\ev}$ and $\beta^{R,M}_{\od}$ as measured by a number,
  \[
\gran_R(M) := \deg\big(\beta^{R,M}_{\ev} - \beta^{R,M}_{\od} \big) + 1  \,,
  \]
that we call the \emph{granularity} of $M$ over $R$. The least value, $\gran_R(M) = 0$, is attained when 
$(\beta^R_i(M))$ is \emph{eventually polynomial}; that is, when $\beta^{R,M}_{\ev} = \beta^{R,M}_{\od}$. 

Our main results link the granularities of $R$-modules and the structure of $R$.  
Let $\gr R$ denote the associated graded ring of $R$ and $\pi\colon \Sym_k(\gri R1) \tra \gr R$ 
the canonical map.  We write $\codim \qr R$ for the height of the ideal generated by the 
quadratic forms in $\Ker(\pi)$ and call that number the \emph{quadratic codimension} of $R$.  

  \begin{Theorem}[Theorem \ref{thm:gran}]
    \label{thm:gran-I}
Every module $M$ over a c.i.\ local ring $R$ satisfies 
  \[
\gran_R(M) \le \max\{\codim R - \codim{\qr R} -1 \,,\,0\} \,.
  \]
   \end{Theorem}

This theorem subsumes a number of contributions, related in time and content as follows.  It was 
proved in \cite{Av:rational} for local rings with $\codim R = \codim{\qr R}$.  When $R$ and $M$ 
are \emph{homogeneous}  (that is, localizations of $\gr R$ and of a graded $\gr R$-module at the 
maximal ideal $(\gr R_1)$) and $\codim R = \codim{\qr R} +1 = 2$, it was obtained by Avramov 
and Zheng \cite{AY} using methods not available in other cases.  Theorem \ref{thm:gran-I} was 
proved for all c.i.\ rings $R$ with $\codim R = \codim{\qr R} +1$ in unpublished joint work 
of the authors of this paper.  Motivated by that result, Avramov, Packauskas, and Walker 
(\cite{APW}, to appear) subsequently proved the full theorem by different techniques.

The proof of Theorem \ref{thm:gran-I}, given below, extends our original approach. It relies on 
the following structure theorem for rings of given quadratic codimension.

\begin{Theorem}[Part of Theorem \ref{thm:PG}]
  \label{thm:PG-I}
A local ring $R$ with infinite residue field has $\codim \qr R = q$ if and only if $\wh R$ is a homomorphic 
image of some c.i.\ local ring $Q$ with $\codim Q = q$, multiplicity $2^q$, and $\edim Q = \edim R$.
   \end{Theorem}

This result is of independent interest, as every local c.i.\ ring of codimension $q$ has multiplicity at
least $2^q$, and those \emph{of minimal multiplicity} are to regular local rings what 
complete intersections of quadrics are to polynomial~rings.

In the second half of the paper we explore possibilities of relaxing the hypothesis or tightening the 
conclusion of Theorem \ref{thm:gran-I}, and completely settle the second issue:

 \begin{Theorem}[Abstracted from Theorem \ref{thm:optimal}]
     \label{thm:thm:optimal-I} 
The upper bound in the inequality in Theorem \emph{\ref{thm:gran-I}} is optimal:  For every pair $(c,q)$ 
of integers with $c \ge q \ge 0$ there exist a c.i.\ local ring $R$ and a cyclic $R$-module $S$ that satisfy
  \[
(\codim R, \codim \qr R) = (c,q) \ \text{ and }\ \gran_R(S) = \max\{c - q -1, 0\} \,.
  \]
  \end{Theorem}

In order to probe the tightness of the hypotheses of Theorem \ref{thm:gran-I} we search for partial 
converses to its statement.  Below we focus on those rings over which all finite modules have 
granularity zero and obtain the following result:

 \begin{Theorem}[Contained in Theorems \ref{thm:strict} and \ref{thm:low}]
     \label{thm:low-I}
Let $(R,\fm,k)$ be a c.i.\ local ring such that the Betti sequence of each finite $R$-module is eventually 
polynomial.

If $R$ is homogeneous, or $\codim R \le 4$ and $k$ is algebraically closed, then one has
  \[
\codim R \le \codim \qr R + 1 \,.
  \]
  \end{Theorem}
  
The proofs of Theorems \ref{thm:thm:optimal-I} and \ref{thm:low-I} hinge upon identifying and constructing families 
of residue rings $S$ of $R$, where the import of invariants of the rings $R$ and $S$ on the values 
of $\gran_R(S)$ can be traced explicitly.  The relevant arguments involve hard computations that 
draw on a number of different techniques.  

The results in this work and in \cite{APW} open up a new narrative concerning the patterns of Betti 
sequences of modules over a given c.i.\ ring.  The methods of proof in these papers suggest possible 
approaches and specific questions.  Here is a sample.

  \begin{Question}
    \label{ques:esspol}
Let $R$ be a c.i.\ local ring $R$ and set $n := \codim R - \codim \qr R$.

Does $R$ have modules with non-polynomial Betti sequences if $n \ge 2$ ?  

Do $R$-modules of maximal granularity, equal to $n - 1$, always exist?
   \end{Question}

We thank Nicholas Packauskas and Mark Walker for numerous useful discussions.

  \section{Complexity and granularity}
     \label{S:Complexity and granularity}

We first overview notation, constructions, and results that will be used throughout the 
main text of the paper.  The statement that $(R,\fm,k)$ is a \emph{local ring} here means 
that $R$ is a commutative noetherian ring with unique maximal ideal $\fm$ and $k$ is the 
residue field $R/\fm$.  As usual, $\dim R$ denotes the (Krull) dimension of $R$ and $\edim R$ 
its \emph{embedding dimension} (that is, the minimal number of generators of~$\fm$); the 
(\emph{embedding}) \emph{codimension} of $R$ is the number $\codim R := \edim R-\dim R$, 
and the number $\codepth R := \edim R-\depth R$ is its (\emph{embedding}) \emph{codepth}.

When $(R',\fm',k')$ is a local ring, a ring homomorphism $\varphi\colon R \to R'$ is \emph{local} 
if $\varphi(\fm)$ lies in $\fm'$.  The map $\varphi$ is faithfully flat if and only if it is flat and local.  
Surjective homomorphisms are assumed to induce the identity map on the residue fields.

For our purposes, it is often convenient to introduce invariants through non-canonical 
presentations of modifications of $R$, or of its $\fm$-adic completion, $\wh R$.

 \begin{chunk}
  \label{ch:regular}
A \emph{regular presentation} of $R$ is a surjective ring map $R\tla P :\! \rho$ with $P$ local 
and regular; we use the same name for an isomorphism $R \cong P/I$, with $P$ as above.  

Every regular presentation $\rho$ factors through one that is \emph{minimal}, meaning 
$\edim P=\edim R$ or, equivalently, $I\subseteq\fp^2$, where $\fp$ is the maximal ideal of~$P$.
Indeed, one has $I/(I\cap\fp^2) \cong \Ker(\fp/\fp^2 \to \fm/\fm^2)$.  Lifting a $k$-basis of this kernel to a 
subset $\bst$ of $I$ yields a minimal presentation $R \cong \ov P/\ov I$, with $\ov P:= P/P\bst$; 
the ring $\ov P$ is regular because $\bst$ extends to a regular system of parameters of $P$. 

By Cohen's Structure Theorem, regular presentations $\wh R \cong P/I$ exist, and such 
\emph{Cohen presentations} also produce minimal ones.  Minimal Cohen presentation need 
not be isomorphic, but $\rank_k(I/\fp I)$ is the same for all of them (see \ref{ch:Betti} below); 
we let $\rel R$ denote that common value and call it the \emph{number of relations} of $R$.
  \end{chunk}

Throughout the paper, $M$ denotes a finite, that is, finitely generated $R$-module.
We review some numerical invariants of minimal free resolutions of modules over local
rings.  For general information on free resolutions we refer to \cite{Av:res}.

  \begin{chunk}
    \label{ch:Betti}
The $i$th \emph{Betti number} $\beta^R_i(M)$ of $M$ is the rank of the $i$th module in a(ny) 
minimal free $R$-resolution $F$ of $M$.  It can be computed in different ways:
  \[
\beta^R_i(M) = \rank_k(F\otimes_Rk) = \rank_k\Tor_i^R(M,k)=\rank_k\Ext^i_R(M,k) \,.
  \]

One measure of the growth of the \emph{Betti sequence} $(\beta^R_i(M))$ is given by the number
  \[
\cx_R(M) := \inf\{n \in \NN_0 \mid \beta^R_i(M) \le a i^{n-1} \text{ for $i\gg0$ and some } a > 0 \}\,,
  \]
called the \emph{complexity} of $M$ over $R$.  Thus $\cx_R(M) = 0$ means that $\pd_RM$ is finite 
and $\cx_R(M) = \infty$ that $i \mapsto \beta^R_i(M)$ cannot be bounded above by a polynomial. 

The Betti numbers of $M$ are handily packed into its \emph{Poincar\'e series}, given by
  \[
P^R_M := \sum_{i \ges 0} \beta^R_i(M) z^i \in \ZZ[[z]] \,.
  \]

If $\wh R \cong P/I$ is a minimal Cohen presentation, then the series $P^P_I$ lies in $\NN_0[z]$
and it is an invariant of $R$; in particular, so is the number $\rank_k(I/\fp I)$; see~\cite[4.1.3]{Av:res}.
   \end{chunk}

We study the asymptotic behavior of a Betti sequence in terms of its Poincar\'e series, complexity,
and granularity---a new invariant that we introduce next.

  \begin{chunk}
    \label{ch:gran}
The sequence $(\beta^R_i(M))$ is \emph{linearly recursive} if and only if the series $P^R_M$ is 
\emph{rational}; that is, if and only if $p \cdot P^R_M$ lies in $\ZZ[z]$ for some nonzero $p\in\ZZ[z]$.

If $P^R_M$ is rational and $\cx_R(M)$ is finite, then the poles of $P^R_M$ are at roots of unity, that 
of highest non-negative order is at $1$, and its order equals $\cx_R(M)$; see \cite[2.4]{Av:asymptotics}.

We say that $M$ has \emph{granularity} $g$ and write $\gran_R(M) = g$ if $P^R_M$ is rational 
and has a pole of order $g \ge 0$ at $-1$.  Formulas involving granularity are stated or used with 
the tacit assumption that the relevant modules have rational Poincar\'e~series.  
   \end{chunk}

The definitions of complexity and granularity given in \ref{ch:gran} and those used in the 
introduction will soon be reconciled; see \ref{ch:rational}.  

 \begin{chunk}
    \label{ch:basics}
The properties of Poincar\'e series and of complexity, listed below, hold without restrictions; 
the formulas for granularity follow from those for Poincar\'e series. 
   \begin{enumerate}[\rm(1)]
   \item
If $N$ is an $n$th syzygy module of $M$ over $R$, then one has 
  \[
P^R_M - z^n P^R_N \in \ZZ[z]\,,\quad \cx_R(M) = \cx_R(N)\,, \quad\text{and}\quad \gran_R(M) =  \gran_R(N) \,.
  \]
   \item
If $R\to (R',\fm',k')$ is a local ring homomorphism, $M'$ denotes the $R'$-module $R'\otimes_RM$, and 
$\Tor^R_i(R',M) = 0$ holds for $i \ge 1$, then one has
  \[
P^R_M = P^{R'}_{M'} \,,\ \  \ \cx_{R}(M) = \cx_{R'}(M')\,, \ \text{ and }\ \, \gran_{R}(M) =  \gran_{R'}(M') \,.
  \]
This is the case, in particular, if $R'$ is flat over $R$, or if $R' = R/R\bsg$ for some $R$-regular set 
$\bsg$ that is also $M$-regular.
   \item
A (\emph{codimension $n$}) \emph{deformation} of $R$ to $Q$ is an isomorphism $R \cong Q/Q\bsf$, with 
$(Q,\fq,k)$ local and $\bsf$ a $Q$-regular set (of $n$ elements); it is \emph{embedded} if $\bsf\subseteq\fq^2$.
We use the same name(s) also for the canonical homomorphism $R \tla Q$.
    \end{enumerate}
    \end{chunk}

Betti sequences whose asymptotic patterns are (almost) completely determined by  
complexity and granularity admit several descriptions:

  \begin{chunk}
    \label{ch:rational}
Let $R$ be a local ring and $M$ a nonzero $R$-module.

The following conditions on an integer $c \ge 0$ are equivalent.
  \begin{enumerate}[\quad\rm(i)]
  \item
There is an inclusion $(1-z^2)^c \cdot P^R_M \in \ZZ[z]$.
  \item
There exists a unique $p^R_M  \in \ZZ[z]$ with $ p^R_M(1) > 0$ such that 
   \[
\qquad P^R_M = \frac{p^R_M}{(1 + z)^{\gran_R(M)}(1 - z)^{\cx_R(M)}} \quad\text{and}\quad 0 \le \gran_R(M) <  \cx_R(M) \le c \,.
   \]
  \item
There exist unique polynomials $\beta^{R,M}_{j} \in \QQ[x]$ for $ j = 0,1$ that 
satisfying the following conditions, where the convention $\deg(0) = -1$ is used:
   \[
      \begin{aligned}
\beta^R_i(M) &= \beta^{R,M}_j(i) \quad\text{for}\quad  i \gg 0 \quad\text{when}\quad i \equiv j\pmod 2 \,,  \quad \text{and}
  \\
\qquad\gran_R(M) &= \deg\big(\beta^{R,M}_{0} - \beta^{R,M}_{1}\big) + 1 <  \deg\big(\beta^{R,M}_{j}\big) +1 = \cx_R(M) \le c\,.  
       \end{aligned}
  \]
  \end{enumerate}

Indeed, it is shown in \cite[Proof of Theorem 4.1]{Av:vpd} that (i) implies (ii).  Partial fraction decomposition 
yields the implications (ii)$\implies$(iii)$\implies$(i).
   \end{chunk}

\begin{chunk}
  \label{ch:ci}
A collection of known results illustrates the conditions in \ref{ch:rational}.
   \begin{enumerate}[\rm(1)]
   \item
If $\pd_Q(\wh M)$ is finite for some codimension $c$ deformation $\wh R \cong Q/Q\bsf$, then 
\ref{ch:rational}(i) holds; see \cite[4.2(i)]{Gu1}; this is a major source of modules whose Poincar\'e 
series have poles only at $\pm1$.  However, $R$-modules $M$ with $P^R_M = 2/(1-z)$ may exist 
over rings $R$ that admit no non-trivial deformations; see \cite{AGP1}.
   \item
The ring $R$ is said to be \emph{complete intersection}, or \emph{c.i.}, if $\wh R$ admits a 
deformation to some regular local ring.  When $\wh R \cong P/I$ is a minimal Cohen presentation, 
$R$ is c.i.\ if and only if $I$ can be minimally generated by some $P$-regular set, if and only if 
$\rel R = \codim R$ (i.e., $I$ can be generated by $\codim R$ elements, see \ref{ch:regular}).
   \item
The following conditions are equivalent: (i) $R$ is c.i.;  (ii) $(1-z^2)^c \cdot P^R_M \in \ZZ[z]$ 
for every $R$-module $M$; (iii) $P^R_k = (1+t)^{\dim R}/(1-t)^{\codim R}$; (iv) $\cx_R(k) < \infty$.
   \end{enumerate}
See (1) for (i)$\implies$(ii), \cite[Theorem~6]{Ta} for (i)$\implies$(iii), and \cite[2.3]{Gu2} for (iv)$\implies$(i).
  \end{chunk}

\begin{chunk}
  \label{ch:adjust}
Maps of local rings $(R,\fm,k)\to (R',\fm',k')$ that are flat with $\fm R' = \fm'$ will be called  
\emph{adjustments} of~$R$; for every $R$-module $M$ and $M' := R' \otimes_R M$ one has 
  \[
H_{M'} = H_{M}\,, \quad \dim_{R'} M' = \dim_{R} M \,, \quad\text{and}\quad \, \edim R' = \edim R'\,.
  \]
 
Any adjustment $R \to (R',\fm',k')$, composed with the completion map $R'\to \wh{R'}$ and some 
minimal Cohen presentation $\wh{R'} \cong P/I$ yields an adjustment $R \to P/I$ with $(P,\fp,k')$ 
regular and $\edim P = \edim R$.  In adjustments $R \to P/I$ with $P$ regular, the presentation 
$P/I \tla P$ is minimal if and only if $\edim P = \edim R$.

Grothendieck \cite[10.3.1]{Gr} proved that every field extension $k \subseteq l$ occurs as the 
extension of residue fields induced by adjustments of $R$, called \emph{inflations}.
  \end{chunk}

\section{Associated quadratic rings}
\label{S:Associated quadratic rings}

Recall that $(R,\fm,k)$ denotes a local ring and $M$ a finite $R$-module. 

In this section we introduce and study invariants of $R$ that are defined in terms of its associated graded ring 
and appear in the main results of the paper.  We first record terminology and notation used when dealing with 
associated graded objects; for general background on graded rings and their modules; see \ref{ch:gradedBetti}.

\begin{chunk}
  \label{ch:HS}
Set $\gr M_j=\fm^jM/\fm^{j+1}M$ for $j\in\ZZ$, and $\gr M=\bigoplus_{j\in\ZZ}\gr M_j$.  Thus, $\gr R$ is 
the \emph{associated graded ring} of $R$ and $\gr M$ the \emph{associated graded $\gr R$-module} of $M$.  
For $x\in M\smallsetminus \{0\}$ set $v(x)=\max\{j\mid x\in\fm^j\}$.  The image of $x$ in $\gr M_{v(x)}$ is called
the \emph{initial form} of $x$ and is denoted by $x^*$; in addition, we set $0^*=0$.

We write $H_{M}$ for $\sum_{j\ges 0}\rank_k{\gri Mj}(z)$.  Since $\gr R$ is generated by $\gri R1$ 
over $\gri R0 =k$, the Hilbert-Serre Theorem yields $h^R_M \in \ZZ[z]$, with $h^R_{M}(1)\ne0$, such that
  \[
H_{M} = h^R_M \cdot (1-z)^{-\dim R} \,.
  \]

The integer $h^R_M(1)$, called the \emph{multiplicity} of~$M$ over~$R$, is denoted by $e_R(M)$;
one has $e_R(M)\ge0$, with equality if and only if $\dim M<\dim R$; set $e(R)=e_R(R)$.
    \end{chunk}

\begin{chunk}
  \label{ch:sym}
Let $\wh R \tla P {\ :\,} \rho$ be a minimal Cohen presentation; see \ref{ch:regular}.  It induces $k$-linear 
isomorphisms $\fp/\fp^2\cong\wh\fp/\wh\fp^2\cong\wh \fm/\wh \fm^2\cong\fm/\fm^2$ that we use to identify 
these vector spaces, and hence their symmetric $k$-algebras.  Thus we view $\gr P$ (cf.\ \ref{ch:HS}) as 
the symmetric $k$-algebra of $\fm/\fm^2$ and we have a canonical surjection $\gr R \tla \gr P {\ :\,} \gr\rho$.  
Set $I^* := \Ker(\gr\rho)$ and call the isomorphism $\gr R \cong \gr P/I^*$ the \emph{canonical presentation 
of $\gr R$}; if a minimal regular presentation $\rho$ (see \ref{ch:adjust}) is at hand, then $I^*$ is equal to 
the ideal of $\gr P$ generated by the set of leading forms $\{f^*\}_{f \in I}$.  

As $P$ is regular and $\dim P=\edim R$, the following relations hold:
  \begin{equation}
    \label{eq:sym1}
\height I^* =\codim \gr R=\codim R = \height I \le \rel R \le \rel \gr R \,.
  \end{equation}
    \end{chunk}

\begin{chunk}
    \label{ch:quad}
We define the \emph{associated quadratic ring} of $R$ to be the graded $k$-algebra
  \begin{equation}
    \label{eq:quad1}
\qr R := {\gr P}/{\qr I} \,,
  \quad\text{where}\quad
\qr I := \gr PI^*_2\,.
  \end{equation}
It is an invariant of $R$, as is the commutative diagram with exact rows
  \begin{equation}
    \label{eq:quad2}
  \begin{aligned}
\xymatrixrowsep{2pc}
\xymatrixcolsep{2pc}
\xymatrix{
0 
\ar@{->}[r]
& \qr I 
\ar@{->}[r]
\ar@{>->}[d]
&\gr P
\ar@{->}[r]
\ar@{=}[d]
&\qr R
\ar@{->}[r]
\ar@{->>}[d]
&0
  \\
0 
\ar@{->}[r]
& I^* 
\ar@{->}[r]
&\gr P
\ar@{->}[r]^-{\gr\rho}
&\gr R
\ar@{->}[r]
&0
}
  \end{aligned}
  \end{equation}
By definition, the ideal $\qr I$ is minimally generated by $\rel \qr R$ quadrics.

If $R \cong P/I$ is a minimal regular presentation, it yields surjective homomorphisms
  \begin{equation}
    \label{eq:quad3}
I\tra I/\fp I\tra I/(\fp^3\cap I)\cong(I+\fp^3)/\fp^3= I^*_2=\qri I2 \,.
  \end{equation}
Letting $\ov f$ denote the class of $f\in I$ in $I/\fp I$ and $\qr f$ its class in $(I+\fp^3)/\fp^3 = \qri I2$, one gets 
a $k$-linear surjection $\ov f \mapsto \qr f$ with $\qr f = f^*$ for $f\notin\fp^3$ and $\qr f = 0$ otherwise.  
  \end{chunk}

For ease of reference, we spell out a few formal properties of that construction.

\begin{Lemma}
  \label{lem:grR}
The ring $\qr R$ and the ideal $\qr I$ from \eqref{eq:quad1} satisfy the relations below.
 \begin{align}
    \label{eq:grR1}
\codim\qr R & = \edim R-\dim\qr R=\height \qr I \le \rel \qr R  \le \rel R \,.
  \\
    \label{eq:grR2}
\codim \qr R & - \codim R = \dim R - \dim\qr R = \height \qr I-\height I^*\le0 \,.
\intertext{\indent If \emph{\ref{ch:sym}} and $R \to (R',\fm',k')$ is an adjustment 
(see \emph{\ref{ch:adjust}}), then}
    \label{eq:grR3}
\codim\qr{R} & = \codim\qr{R'} 
\text{ and }
\rel\qr{R} = \rel\qr{R'}\,.
\intertext{\indent If $R \tla (Q,\fq,k)$ is a surjective ring homomorphism with kernel in $\fq^2$, then}
    \label{eq:grR4}
\codim\qr{R} &\ge \codim\qr{Q}
\quad\text{and}\quad
\rel \qr R \ge \rel \qr Q \,.
  \end{align} 
 \end{Lemma}

   \begin{proof}
In \eqref{eq:grR1} the equalities hold because $\gr P$ is a polynomial ring; in \eqref{eq:grR2} they follow
from \eqref{eq:grR1} and \eqref{eq:sym1}.  The Principal Ideal Theorem, the surjection \eqref{eq:quad3},
and $\qr I \subseteq I^*$ yield the inequalities in \eqref{eq:grR1} and \eqref{eq:grR2}.  

Let $\gr{P'}$ denote the symmetric $k'$-algebra of $\fm'/\fm'^2$; since $R\to R'$ induces isomorphisms 
$\gr R\otimes_kk' \cong \gr{R'}$ and $\gr P\otimes_kk' \cong \gr{P'}$ of graded $k'$-algebras, $(?\otimes_kk')$ 
turns \eqref{eq:quad2} into the corresponding diagram for $R'$, and \eqref{eq:grR3} follows.  

In particular,
for \eqref{eq:grR4} we may assume that $R$ is complete.  A minimal Cohen presentation $\wh Q \cong P/J$ 
then yields such a presentation $R\cong P/I$.  As $\gr P \tra \gr R$ factors through $\gr P \tra \gr Q$, we get 
$I^* \supseteq J^*$, whence $\qr I \supseteq \qr J$, and \eqref{eq:grR4} follows.
   \end{proof}

The largest value of $\codim \qr R$ allowed by \eqref{eq:grR1} is $\edim R$: it is reached if and only 
if $\qr I=0$; that is, if and only if $I$ lies in $\fp^3$.  We have no similar description of the rings with 
$\codim \qr R = \codim R$, the largest value allowed by \eqref{eq:grR2}, except if $R$ is \emph{c.i.}; 
see Proposition \ref{prop:minmult}; we record a few facts used in its proof, and later on.

\begin{chunk}
  \label{ch:VV}
Let $(P,\fp,k)$ be a local ring and $Q := P/(g_1,\dots,g_s)$ with $g_i\in\fp^{n_i}$ for $1\le i\le s$.
\begin{enumerate}[\rm(1)]
  \item
The set $\{g^*_1,\dots,g^*_s\}$ is $\gr P$-regular if and only if it generates $\Ker(\gr P\to\gr Q)$
and $\{g_1,\dots,g_s\}$ is $P$-regular; see Valabrega and Valla \cite[2.7 and 1.1]{VV}.
  \item
When $\{g_1,\dots,g_s\}$ is part of a system of parameters, there is an inequality
  \[
e(Q)\ge n_1\cdots n_s\cdot e(P) \,;
  \]
equality holds if $\{g^*_1,\dots,g^*_s\}$ is $\gr P$-regular; see \cite[VIII, \S7, Proposition 4]{Bo}.
  \item
Assume $\gr P$ is Cohen-Macaulay.  The set $\{g^*_1,\dots,g^*_s\}$ is $\gr P$-regular if and only 
if $\{g_1,\dots,g_s\}$ is $P$-regular and equality holds in (2) above; such an equivalence is 
proved by Rossi and Valla \cite[1.8]{RV} under the additional hypothesis that $\{g_1,\dots,g_s\}$ 
is regular, which is superfluous in one direction, due to (1) above.
  \item
When $Q$ is c.i.\ and $\wh Q \tla P$ is a minimal Cohen presentation, (2) above yields 
 \[
e(Q) = e(\wh Q) \ge v(g_1)\cdots v(g_s)\cdot e(P) \ge  2^{\codim \wh Q} = 2^{\codim Q}\,.
  \]
When $e(Q) = 2^{\codim Q}$ holds the ring $Q$ is said to be \emph{c.i.\ of minimal multiplicity}.
  \item
The ring $Q$ is c.i.\ (of minimal multiplicity), if some, and only if all of its adjustments have the 
corresponding property.  Indeed, the invariants used to define these notions do not change when  
$Q$ is replaced by an adjustment; see~\ref{ch:HS}.
 \end{enumerate}
   \end{chunk}

Next we collect various characterizations of local c.i.\ rings of minimal multiplicity.  
They are used in upcoming proofs to produce and/or to recognize such rings.
For notions and notation concerning graded rings we refer to Section \ref{S:A family of Golod homomorphisms}.

\begin{Proposition}
  \label{prop:minmult}
The following conditions on a local ring $Q$ are equivalent.
  \begin{enumerate}[\quad\rm(i)]
 \item
$Q$ is c.i.\ of minimal multiplicity (see \emph{\ref{ch:VV}(5)}). 
 \item
$Q$ is c.i.\ and $\codim Q=\codim \qr Q$. 
 \item
$Q$ is c.i.\ and $\gr Q\cong\qr Q$ as graded $k$-algebras.
 \item
$Q$ is c.i.\ and the graded $k$-algebra $\gr Q$ is Koszul (see \emph{\ref{ch:gradedBetti}(2)}).
 \item
$\gr Q$ is a graded complete intersection of quadrics (see \emph{\ref{ch:gradedBetti}(3)}).
 \item
If $\wh Q \cong P/J$ is a minimal Cohen presentation and $\{g_1,\dots,g_s\}$ minimally generates 
$J$, then $\{\qri g1, \dots, \qri gs\}$ minimally generates $J^*$ and is $\gr P$-regular.
   \end{enumerate}
  \end{Proposition}

  \begin{proof}
We set $d := \edim Q$ and assume, as we may (see \ref{ch:VV}(4)) that $Q$ is complete.  

(i)$\implies$(vi).
With $n_i : = v(g_i)$ for $1\le i\le s$ in \ref{ch:VV}(4), we get $n_i=2$, and hence $g_i^* = \qri gi$.  
Thus $\{\qri g1, \dots, \qri gs\}$ is $\gr P$-regular by \ref{ch:VV}(3) and generates $J^*$ by \ref{ch:VV}(1).

(vi)$\implies$(v).
This implication follows from the hypothesis, as $\gr Q = \gr P/J^*$.

(v)$\implies$(iv).
The ring $Q$ is c.i., by \ref{ch:VV}(1).   From $P^{\gr Q}_k = (1+yz)^e/(1 - y^2z^2)^s$ (see \ref{ch:gradedBetti}(3)),
we get $\beta^{\gr Q}_{i,j}(k) = 0$ for $j \ne i$; therefore $\gr Q$ is a Koszul algebra.

(iv)$\implies$(i).
As $\gr Q$ is Koszul, $\sum_j \beta^{\gr Q}_{i,j}(k) = \beta^Q_i(k)$ holds for every integer $i$; 
see \c Sega \cite[2.3]{S1}.  With $s := \codim Q$, this result yields the third equality in the string
  \[
\frac{1}{H_{\gr Q}(-z)}
= \frac{H_k(-z)}{(-z)^{0}H_{\gr Q}(-z)}
= P^{\gr Q}_k(1,z)
= P^{Q}_k(z)
= \frac{(1+z)^{\dim Q}}{(1-z)^{s}} \,.
  \]
The second one comes from \ref{ch:linear}(1), and the fourth from \ref{ch:ci}(3).  
Thus we get equalities $H_{Q}(y) = H_{\gr Q}(y) = (1+y)^{s}/(1-y)^{\dim Q}$, whence $e(Q) = 2^{\codim Q}$.

(vi)$\implies$(iii).
This implication is given by \ref{ch:VV}(1).

(iii)$\implies$(ii).
This implication holds because $\codim Q = \codim \gr Q$.

(ii)$\implies$(i).
The hypothesis and Formulas \eqref{eq:sym1} and \eqref{eq:grR1} yield (in)equalities
  \[
\codim \gr Q = \codim Q = \codim \qr Q \le \rel \qr Q \le \rel Q = \codim Q 
  \]
that force equalities throughout.  In particular, $\qr Q$ is a graded complete intersection of 
$s := \codim Q$ quadrics; from the surjection $\qr Q \tra \gr Q$ and \ref{ch:VV}(5), we obtain 
  \[
2^{\codim Q} = 2^s = e(\qr Q) \ge e(\gr Q) = e(Q) \ge 2^{\codim Q} \,.
  \qedhere
  \]
   \end{proof}

We are ready for the main results in this section, which concern general local rings.  In the special 
case when $k = k'$ and $R' =R'' = \wh R$, the first theorem below yields a structure theorem for 
local rings with prescribed quadratic codimension, stated in the introduction as Theorem \ref{thm:PG-I}.  
The second theorem provides, under manageable additional hypotheses, families of local rings with 
prescribed quadratic codimension parametrized by dense subsets of affine spaces.
  
\begin{Theorem}
  \label{thm:PG}
Let $(R,\fm,k)$ be a local ring, and set $r := \rel R$ and $q := \codim \qr R$.
  \begin{enumerate}[\rm(1)]
  \item
For each field extension $k \hra k'$ with $k'$ infinite there exists an adjustment $R\to P/I$ 
with $(P,\fp,k')$ regular, $I = (f_1,\dots,f_r)$, and $Q := P/(f_1,\dots,f_q)$ c.i.\ of minimal multiplicity; 
every such adjustment satisfies $\edim P = \edim R$.
   \item
If $R \to R'$ is an adjustment, $R' \tla R'' \tla Q$ are surjective ring maps, and $Q$ is local 
c.i.\ of minimal multiplicity with $\codim Q = q$ and $\edim Q = \edim R$, then
 \[
 \codim \qr {R} = \codim \qr {R''} = \codim \qr Q  = \codim Q \,.
  \]
  \end{enumerate}
   \end{Theorem}

  \begin{proof}
(1) Referring to \ref{ch:adjust}, choose an adjustment $R\to R' = P/I$ with $k \hra k'$ the induced 
residue field extension and $P \tra P/I$ a minimal regular presentation.  Due to the equalities
$\height \qr I = \codim \qr{R'} = \codim \qr R = q$ and $\rel \qr{R'} = \rel R = r$ (see \eqref{eq:grR1} and 
\eqref{eq:grR3}), the ideal $\qr I$ is minimally generated by $r$ elements and contains $\gr P$-regular 
sets of $q$ forms.  As $\qr I$ is generated by quadrics and $k'$ is infinite, $\qri I2$ contains a 
$\gr P$-regular set of $q$ elements.  In view of \eqref{eq:quad3}, it can be chosen in the form 
$\{\qri f1,\dots, \qri fq\}$ with $f_i \in I$; by (vi)$\implies$(i) in Proposition \ref{prop:minmult}, the
ring  $P/(f_1,\dots,f_q)$ is c.i.\ of minimal multiplicity.  Since $\{\qri f1,\dots, \qri fq\}$ is $k$-independent,
$\{f_1,\dots,f_q\}$ can be extended to a minimal set of generators of $I$. 

(2)  The hypothesis provides the equalities that bookend the following string:
 \[
q = \codim \qr R  = \codim \qr {R'} \ge \codim \qr {R''} \ge \codim \qr Q = \codim Q = q  \,.
  \]
For the rest, use \eqref{eq:grR3}, \eqref{eq:grR4}, and (i)$\implies$(ii) in Proposition \ref{prop:minmult}. 
   \end{proof}

  \begin{Theorem}
 \label{thm:family}
Let $(P,\fp,k)$ be a local ring; let $\ov a$ denote the image in $k$ of $a \in P$.  
Given $(f_1,\dots,f_r)\in(\fp^2)^r$ and
$\bsa := (a_1,\dots,a_{r-1})\in P^{r-1}$ set $f^{\bsa}_i :=f_i-a_if_r$ for $1 \le i \le r-1$.  Put
$\ov\bsa := (\ov a_1,\dots,\ov a_{r-1})$, where $\ov a$ is the image of $a$ in $k$.

If $k$ is algebraically closed, $\{f_1,\dots,f_q,f_r\}$ is $P$-regular for some $q < r$, and 
$\{\qri f1,\dots,\qri fq\}$ is $\gr P$-regular, then the following set is  Zariski-open and not empty:
  \[
\CU :=  \{\ov\bsa\in \mathbb{A}^{r-1}_{k} \mid P/(f^{\bsa}_1,\dots,f^{\bsa}_q) \text{ is c.i.\ of minimal multiplicity}\,\}\,.
  \]
  \end{Theorem}
 
   \begin{proof}
Let $k[\bsx]$ be the polynomial ring with indeterminates $x_1,\ldots,x_q$. 

The ring $P/(f^{\bsa}_1,\dots,f^{\bsa}_q)$ is c.i.\ of minimal multiplicity if and only if the set 
$\qri{\bsf}{\bsa} := \{\qri f1-\ov a_1\qri fr,\ldots,\qri fq-\ov a_q\qri fr \}$ is $\gr P$-regular; see (i)$\iff$(vi) in 
Proposition \ref{prop:minmult}.  The set $\qri{\bsf}{\bsa}$ is regular if and only if $\dim \gri P{\bsa} \le d-q$ 
holds with $\gri P{\bsa} := \gr P/\gr P\qri{\bsf}{\bsa}$ and $d:=\dim P$. The algebra $\gri P{\bsa}$ is the fiber 
of the canonical map
  \[
k[\bsx] \to (k[\bsx]\otimes_k \gr P)/ (1\otimes\qri f1-x_1\otimes\qri fr,\ldots, 1\otimes\qri fq - x_q\otimes\qri fr) 
  \]
at the maximal ideal $\fn_{\bsa} := (x_1 - \ov a_1,\dots,x_q - \ov a_q)$.  Since fiber dimension is upper 
semicontinuous (see \cite[14.8.b]{Ei2}), $V := \{\ov\bsa\in \mathbb{A}^q_{k} \mid \dim(P_{\bsa}) > d - q\}$ 
is closed in $\mathbb{A}^q_{k}$. Thus the set $\CU$ is open, as it equals ${A}^q_{k} \setminus V$, and 
$\CU$ contains $\boldsymbol 0$ by hypothesis.
 \end{proof}

Free resolutions over c.i.\ rings of minimal multiplicity are known to have special properties.  
We note two, which will be promptly applied in the next section.

  \begin{chunk}
    \label{ch:base_case}
If $Q$ is a local c.i.\ ring of minimal multiplicity and $N$ a $Q$-module, then one has 
$P^Q_N = p^Q_N(z) \cdot (1-z)^{-\cx_Q(N)}$ with $p^Q_N(z) \in \ZZ[z]$ and 
$p^Q_N(1) \ne 0$; see \cite[2.3]{Av:rational}.
  \end{chunk}

   \begin{Proposition}
  \label{prop:Sega}
With notation as in \emph{\ref{ch:base_case}}, $\dim N<\dim Q$ implies $p^Q_N(-1) = 0$.
    \end{Proposition}

   \begin{proof}
As $\gr Q$ is Koszul (see Proposition \ref{prop:minmult}), the module $N$ has finite linearity 
defect; see Herzog and Iyengar, \cite[5.10]{HI}.  Thus \c Sega' s result \cite[6.2]{S2} applies and, in 
view of the expression for $P^Q_N$ in \ref{ch:base_case}, it yields the first equality in the string
   \[
p^Q_N(-1)e(Q)=(1-(-1))^{\cx_Q(N)}e_Q(N)=2^{\cx_Q(N)}e_Q(N) \,.
  \]

Since $\dim N<\dim Q$ means $e_Q(N)=0$, we obtain $p^Q_N(-1)=0$, as desired.
   \end{proof}
 
\section{An upper bound on granularity}
\label{S:An upper bound on granularity}

Here our goal is to give a concise proof of Theorem \ref{thm:gran-I}.

  \begin{Theorem}
  \label{thm:gran}
Every finite module $M$ over a c.i.\ local ring $(R,\fm,k)$ satisfies 
    \begin{equation}
    \label{eq:gran1}
\gran_R(M) \le \begin{cases}
\codim R - \codim{\qr R} -1 &\text{if}\quad \codim R  \ge \codim{\qr R} + 2\,;
  \\
0 & \text{if}\quad \codim R \le \codim{\qr R} + 1\,.
 \end{cases}
  \end{equation}
   \end{Theorem}

As noted in the introduction, a different proof Theorem \ref{thm:gran} was obtained in \cite{APW}.  
Our argument, presented in \ref{ch:pfPG}, proceeds by induction on $n := \codim R - \codim{\qr R}$; 
see \eqref{eq:grR2}.  For this we need yet another way of factoring embedded deformations through 
deformations with specified properties, provided by the next result.

\begin{Theorem} 
    \label{thm:def}
Let $(P,\fp,k)$ be a local ring, $\bsf := \{ f_1,\dots,f_c \}$ a $P$-regular set contained in $\fp^2$, and 
put $R := P/P \bsf$.  For every $\bsa := (a_1,\dots,a_{c-1}) \in P^{c-1}$, set $f^{\bsa}_j := f_j-a_jf_c$ for 
$1 \le j \le c-1$ and put $\bsf_{\bsa} := (f^{\bsa}_1,\dots,f^{\bsa}_{c-1})$ and $R_{\bsa} := P/P \bsf_{\bsa}$.

Let $N$ be an $n$th $R$-syzygy module of a finite $R$-module $M$ that satisfies
  \[
\pd_{P}(M) < \infty = \pd_R(M) \,.
  \]

There exist an integer $\crdeg_R(M) \ge -1$ and a finite set $Z(M)$ of linear varieties in $\mathbb{A}^{c-1}_k$, 
defined in \eqref{eq:pfdef2}, such that the following conditions are~equivalent: 
  \begin{enumerate}[\quad\rm(i)]
   \item
$P^{R}_N = P^{R_{\bsa}}_N \cdot (1-z^2)^{-1}$.
   \item
$n > \crdeg_R(M)$ and $\ov\bsa \notin \bigcup_{V \in Z(M)} V$,
where $\ov\bsa$ is the image of~$\bsa$ in $\mathbb{A}^{c-1}_k$.
  \end{enumerate}
If $k$ is infinite, then $\bigcup_{V \in Z(M)} V \neq \mathbb{A}^{c-1}_k$. 
  \end{Theorem}

 \begin{remark}
    \label{rem:ops}
The prototype of Theorem \ref{thm:def} is \cite[Theorem 3.1]{Ei1}, and both proofs utilize rings of cohomology 
operators defined by the deformation $R \tla Q$.  Such a structure was introduced by Gulliksen \cite{Gu1},
and theories with similar properties were produced in \cite{Me, Ei1, Av:vpd, AS, AGP2, AB} from a priori 
incomparable constructions.

We present the proof in detail because the argument for \cite[3.1]{Ei1} is incomplete and references 
to several sources are needed to fill in the gaps. The facts that we use are listed below, along with 
pointers to the earliest published proof.
    \end{remark}

  \begin{chunk}
    \label{ch:actions}
The hypotheses in the opening sentence of Theorem \ref{thm:def} are in force.

Let $X$ denote the $k$-vector space $P\bsf/\fm \bsf$ and $\ov\bsf$ its basis $\{\ov f_1,\dots, \ov f_c\}$,
consisting of the images of $f_j$ for $j = 1,\dots,c$.  Let $\{\chi_1,\dots, \chi_c\}$ be the dual basis of the
vector space $\CX := \Hom_k(X,k)$ and let $\CR$ be the symmetric algebra of the graded vector space 
that has $\CX$ in degree $2$ and $0$ in all other degrees.  We identify $\CR$ and the graded polynomial 
ring $k[\chi_1,\dots, \chi_c]$ with indeterminates of degree $2$.

The graded vector space $\Ext_R(M,k) := \bigoplus_{i \ges 0} \Ext^{i}_R(M,k)$ supports a structure of 
graded left $\CR$-modules that has the following properties: 
   \begin{enumerate}[\rm(1)]
  \item
The assignment $? \rightsquigarrow \Ext_R(?,k)$ is a contravariant additive functor from the category of 
$R$-modules to that of graded $\CR$-modules; see \cite[Theorem 3.1(i)]{Gu1}. 
  \item
The connecting maps in cohomology sequences induced by short exact sequences of $R$-modules 
commute with the actions of $\CR$; this holds as \cite[Theorem, p.\,700]{AS} shows that 
the action of $\CR$ factors through Yoneda products. 
  \item
For $R' := P/(f_1,\dots, f_{c-1})$ and for every $i \in \ZZ$ there is an exact sequence 
   \[
\Ext^{i-2}_{R}(M,k) \xra{\chi_c} \Ext^{i}_{R}(M,k) \to \Ext^i_{R'}(M,k) 
\to \Ext^{i-1}_{R}(M,k) \xra{\chi_c} \Ext^{i+1}_{R}(M,k) 
   \]
See \cite[Theorem 2.3]{Av:vpd}; it is implicit in \cite[Formula (8) on p.~178]{Gu1}. 
  \item
The $\CR$-module $\Ext_R(M,k)$ is finitely generated if the $R$-module $M$ is finite with $\pd_P M < \infty$;  
see \cite[Theorem 3.1(ii)]{Gu1}.  The proof of \cite[Theorem 3.1]{Ei1} is flawed: it uses
\cite[Proposition 1.6]{Ei1} whose proof is invalid; see \cite[Remark~4.2]{AS}.
  \end{enumerate}
The properties of $\CR$-modules described above can be used concurrently because the results 
in \cite[Section 4]{AS} show that the sets of operators produced by the constructions in 
\cite{Me, Ei1, Av:vpd, AS, AGP2, AB} differ at most by some sign.
  \end{chunk}

At a final stop before the proof in \ref{ch:pfdef} we introduce notation to describe $Z(M)$. 

 \begin{chunk} 
    \label{ch:perp}
Let $k$ be a field, $\CX$ a $k$-vector space, and $\{\chi_j\}_{1 \les j\les c}$ a basis of $\CX$.  

Let $\CV \subseteq \CX$ be a subspace of rank $d$.  If $\CV \ne 0$ let $\big\{\sum_{j=1}^c a_{l,j}\chi_j \big\}_{1\les l \les d}$ 
with $a_{l,j} \in k$ be a basis of~$\CV$; let $k[\bsx] := k[x_1,\dots,x_{c-1}]$ be a polynomial ring, put
  \[
A(\bsx) := 
\begin{pmatrix}
a_{1,1}     & a_{1,2}    & \dots   & a_{1,d}     & x_1        \\
a_{2,1}     & a_{2,2}    & \dots   & a_{2,d}     & x_2        \\
 \vdots      &  \vdots     & \ddots &  \vdots     &  \vdots   \\
a_{c-1,1}  & a_{c-1,2} & \dots   & a_{c-1,d}  & x_{c-1}  \\
a_{c,1}     & a_{c,2}     & \dots   & a_{c,d}    & 1           \\
\end{pmatrix} 
  \]
and let $\CV^{\perp}$ be the zero set in $\mathbb{A}^{c-1}_k$ of the maximal minors of $A(\bsx)$ that contain the 
last column; also, set $0^{\perp} := \emptyset$.  For $\bsu := (u_1,\dots,u_{c-1}) \in \mathbb{A}^{c-1}_k$, one has:
  \begin{equation}
    \label{eq:perp1}
\bigg[ \sum_{j=1}^{c-1} u_i\chi_j +\chi_c \in \CV\bigg] \iff \big[\rank_k A(\bsu) = d \,\big] \iff \big[\bsu \in \CV^\perp\big] \,.
  \end{equation}

If these conditions hold, then $\CV^\perp$ is a linear subvariety of $\mathbb{A}^{c-1}_k$.
   \end{chunk}

  \begin{chunk}
  \label{ch:pfdef}
\emph{Proof of Theorem \emph{\ref{thm:def}}}.
We keep the notation in the statement of the theorem.

A minimal free resolution $F$ of $M$ yields an exact sequence of $R$-modules
  \begin{equation}
    \label{eq:pfdef1}
0 \to N \to F_{n-1} \to \cdots \to F_0 \to M \to 0
  \end{equation}
finite free $F_i$'s.  Thus $N$ satisfies $\pd_{R}(N) = \infty > \pd_P(N)$, and hence 
   \[
\CM := \bigoplus_{i \in \ZZ} \Ext^i_R(M,k) \quad\text{and}\quad \CN := \bigoplus_{i \in \ZZ} \Ext^i_R(N,k)
  \]
 are finitely 
generated graded left $\CR$-modules; see \ref{ch:actions}(4).  As $\bsf_{\bsa} \cup \{f_c\}$ minimally generates 
$P\bsf$, the set $\{\ov {f^{\bsa}_1},\dots,\ov {f^{\bsa}_{c-1}}\}\cup\{\ov {f_c}\}$ is a $k$-basis of the vector space 
$X$ defined in \ref{ch:actions}.  The dual basis of the space $\Hom_k(X,k) = \CR_{2}$ is the set
$\{\chi^{\ov\bsa}_1,\dots,\chi^{\ov\bsa}_c\}$ of the  
$\CX := \CR_{2}$ has $\chi^{\ov\bsa}_j = \chi_j$ for $j \le i \le c - 1$ and $\chi^{\ov\bsa}_c = \sum_{j=1}^{c-1} \ov a_j\chi_j +\chi_c$.   

Put $R_{\bsa} := P/P \bsf_a$.  From \ref{ch:actions}(3) we get an exact sequences of $k$-vector spaces
  \[
0 \to \CK_{i-2} \to \Ext^{i-2}_{R}(N,k) \to \Ext^i_{R}(N,k) \to \Ext^i_{R_{\bsa}}(N,k) \to \CK_{i - 1} \to 0
  \]
with $\CK_i := \{\nu \in \Ext^i_{R}(N,k) \mid \chi^{\ov\bsa}_c \nu = 0\}$ for $i \in \ZZ$.  The resulting equality
  \[
(1 - z^2) P^R_N = P^{R_{\bsa}}_N  - (1 + z) \sum_{i \ges 0} \rank_k \CK^i z^i 
  \]
shows that condition (i) in the theorem is equivalent to the following condition:
  \begin{enumerate}[\quad\rm(i)]
   \item[\rm{(i$'$)}]
$\chi^{\ov\bsa}_c = \sum_{j=1}^{c-1} \ov a_j\chi_j +\chi_c$ is $\CN$-regular.
  \end{enumerate}

The iterated connecting maps $\Ext^i_{R}(N,k) \to \Ext^{i+n}_{R}(M,k)$ defined by the exact 
sequence \eqref{eq:pfdef1} are bijective.  In view of \ref{ch:actions}(2) they coalesce into an 
isomorphism $\CN \cong \CM_{\ges n}(n)$ of graded $\CR$-modules.  Put 
$\Ass^{\circ}_{\CR}(\CM) := \Ass_{\CR}(\CM) \setminus \{ \CR_{>0} \}$, and also
 \begin{equation}
  \label{eq:pfdef2}
 \begin{aligned}
\crdeg_R(M) & := \sup\{i \in \ZZ \mid \Ann_{\CR}(\mu) = \CR_{>0} \text{ for some } \mu \in \CM_{i} \} \,;
  \\       
Z(M) &:= \{(\CP_2)^{\perp} \subseteq \mathbb{A}^{c-1}_k \mid \CP \in \Ass^{\circ}_{\CR}(\CM) \}
\text{ with } ?^\perp \text{ defined in } \ref{ch:perp} \,.
       \end{aligned}
       \end{equation}

We complete the proof of the theorem by showing that (i$'$) is equivalent to~(ii).

The number $t := \crdeg_R(M)$ is an integer because $(0:_{\CM} \CR_{>0})$ is a homogeneous 
subspace of $\CR$ and has finite $k$-rank.  The following statements are equivalent:
  \begin{equation}
    \label{eq:pfdef3}
\big[\Ass_{\CR}(\CM_{\ges n}) = \Ass^{\circ}_{\CR}(\CM) \big] 
 \iff \big[\CR_{>0} \notin \Ass_{\CR}(\CM_{\ges n})\big] 
 \iff \big[ n > t \big] \,.
  \end{equation}
Indeed, we have $\Ass_{\CR}(\CM_{\ges n}) \subseteq \Ass_{\CR}(\CM) \subseteq \Ass_{\CR}(\CM_{\ges n})\cup \{\CR_{>0}\}$ 
as $\CM/\CM_{\ges n}$ is of finite length; the implication $\Leftarrow$ on the left follows, the rest hold by definition.

As $\CM/\CM_{< t}$ is not zero, we have $Z := Z(M) \ne \emptyset$.  Condition (i$'$) is equivalent to 
$\chi^{\ov\bsa} \notin \bigcup_{\CP\in \Ass_{\CR}(\CM_{\ges n})} \CP$, and hence to $n > t$ and 
$\chi^{\ov\bsa} \notin \bigcup_{\CP\in Z} \CP$, due to \eqref{eq:pfdef2}. The last exclusion is equivalent to 
$\chi^{\ov\bsa} \notin \bigcup_{\CP\in Z} \CP_{2}$ (as each $\CP$ is homogeneous), which can be rewritten 
as $\ov\bsa \notin \bigcup_{\CP \in Z} (\CP_{2})^\perp$, by \eqref{eq:perp1}.  Finally, recall that affine spaces 
over infinite fields are not unions of finitely many proper linear subvarieties.  
 \qed
   \end{chunk}

  \begin{remark}
    \label{rem:incomplete}
A \emph{critical degree} is defined in \cite[7.1]{AGP2} for every nonzero finite module over any local ring
in terms of chain endomorphisms of its minimal free resolutions; in the context of Theorem \ref{thm:def} 
it is equal to the integer in \eqref{eq:pfdef2}; see \cite[7.2(1)]{AGP2}. 

A priori estimates for the critical degree are known in case $\pd_Q M$ is finite for some
deformation $R \tla Q$; they involve the number $g := \depth R - \depth_RM$:
   \begin{itemize}
   \item
$\crdeg_R(M) = g$ if $\cx_R M \le 0$, by the Auslander-Buchsbaum Equality. 
   \item
$\crdeg_R(M) \le g$ if $\cx_R M = 1$; see \cite[5.3 and 6.1]{Ei1} and \cite[7.3(1)]{AGP2}.
   \item
$\crdeg_R(M) \le g + \max\{2\beta^R_g -2, 2\beta^R_{g+1} -1\}$ if $\cx_R M = 2$; see \cite[7.6]{AB}).
   \end{itemize}

The last assertion of Theorem \ref{thm:def} may fail when $k$ is finite; see \cite[6.7]{Av:vpd}.\smallskip
  \end{remark}

  \begin{chunk}
  \label{ch:pfPG}
\emph{Proof of Theorem \emph{\ref{thm:gran}}}.
Set $c := \codim R$, $q := \codim{\qr R}$, and $n := c - q$.

We argue by induction on $n$.  When $n = 0$ the ring $R$ has minimal multiplicity (see Proposition
\ref{prop:minmult}), and then \ref{ch:base_case} yields $\gran_R M = 0$; this is the desired result.  

Now we assume $n \ge 1$ and set out to prove $\gran_R(M) < n$.  The invariants in play do not change 
under adjustments of $R$; see \ref{ch:basics}(2) and \eqref{eq:grR3}.  Due to Theorem \ref{thm:PG}(1) 
we may assume $k$ algebraically closed and $R = P/P\bsf$ for some regular local ring $(P,\fp,k)$ and 
$P$-regular sequence $\bsf: = (f_1,\ldots,f_{c})$ contained in $ \fp^2$ such that 
$Q := P/(f_1,\ldots,f_q)$ is a c.i.\ ring of minimal multiplicity and $\codim \qr Q =  q < c$. 

For every $\bsa := (a_1,\dots,a_{c-1})\in P^{c-1}$ and $1 \le i \le c-1$, put $f^{\bsa}_i :=f_i-a_if_c$.
The deformation $R \tla P$ factors as a composition of deformations
  \[
R \tla R_{\bsa} := P/(f^{\bsa}_1,\dots,f^{\bsa}_{c-1}) \tla
Q_{\bsa} := P/(f^{\bsa}_1,\dots,f^{\bsa}_{q}) \tla P \,.
  \]
Let $\ov \bsa$ be the image of $\bsa$ in $\mathbb{A}^{c-1}_k$.  Theorem \ref{thm:def} yields an 
$R$-syzygy module $N$ of $M$ and a non-empty Zariski-open set $\CU_1$ of 
$\mathbb{A}^{c-1}_k$ such that $\ov\bsa\in \CU_1$ implies  
    \begin{equation}
  \label{eq:pfPG1}
P^{R}_N = P^{R_{\bsa}}_N \cdot (1-z^2)^{-1} \,.
    \end{equation}
Theorem \ref{thm:family} produces a non-empty Zariski-open set $\CU_2 \ne \emptyset$ of $\mathbb{A}^{c-1}_k$ such 
that for each $\ov\bsa\in \CU_2$ the ring $Q_{\bsa}$ is c.i.\ of codimension $q$ and minimal multiplicity.  Note that 
$\CU_1 \cap \CU_2$ is not empty and choose $\bsa$ with $\ov\bsa \in \CU_1 \cap \CU_2$.

If $n = 1$, then $R_{\bsa} = Q_{\bsa}$ and $\dim N  < \dim R_{\bsa}$ hold and we obtain
    \[
P^{R}_N = \frac{P^{R_{\bsa}}_N}{(1-z^2)} 
= \frac{(1+z)\cdot p(z)}{(1-z)^{\cx_{R_{\bsa}}(N)} \cdot (1-z^2) }  
=  \frac{p(z)}{(1-z)^{\cx_R(N) + 1}}
    \]
from \eqref{eq:pfPG1}  and Proposition \ref{prop:Sega}.  This gives $\gran_{R} N = 0$, and hence $\gran_{R} M = 0$ 
(see \ref{ch:basics}(1)); thus \eqref{eq:gran1} holds for $n=1$.  When $n\ge 2$ we may suppose, by induction, that 
\eqref{eq:gran1} holds for local rings $S$ with $\codim S - \codim{\qr{S}} < n$.  Referring to \ref{ch:basics}(1), 
\eqref{eq:pfPG1}, the induction hypothesis, and Theorem \ref{thm:PG}(2) we get
   \begin{align*}
\gran_{R} M = \gran_{R} N  \le \gran_{R_{\bsa}} N + 1  < \codim R_{\bsa} - \codim \qr{(R_{\bsa})} + 1 = c - q \,.
   \end{align*}

The induction step is complete, and with it the proof of \eqref{eq:gran1}.  
  \qed
   \end{chunk}

\section{The upper bound is optimal}
\label{S:The upper bound is optimal}

In this section we prove that the upper bound on granularity, established in Theorem 
\ref{thm:gran}, cannot be tightened in general; see Theorem \ref{thm:optimal} below.  

\begin{chunk}
     \label{ch:varphi}
Let $\varphi \col (R,\fm,k)\tra (S,\fn,k)$ be a surjective homomorphism of local rings.  
Choose a minimal Cohen presentation $\rho\col (P,\fp,k)\tra\wh R$. Put $I:=\Ker(\rho)$ 
and $\wJ := \Ker(\wh\varphi\rho)$, and choose a subset $\bst$ of $\wJ$ that is mapped 
bijectively onto some $k$-basis of $\wJ/\wJ\cap\fp^2$.  Put $(Q,\fq,k):=(P/P\bst,\fp/P\bst,k)$
and choose in $\fp$ a subset that is is mapped bijectively onto some minimal set of generators $\fq$.  
The exact sequence 
 \[
0 \to \wJ/\wJ\cap\fp^2 \to \fp/\fp^2 \to \fq/\fq^2 \to 0
  \]
of $k$-vector spaces shows that $\bst\sqcup\bsu$ minimally generates $\fp$.  Thus $\wh S \cong Q/J$ 
with $J := P/ I$ is a minimal Cohen presentation.  As $P^Q_{J}$ is an invariant of $S$ (cf.\ \ref{ch:Betti}),
so is the first integer defined below; the second one is an invariant of $\varphi$ (see~\cite{Av:golod}):
  \[
m(S) := \max \{n \in \NN_0 : (1+z)^n \, |\, (z^2 P^Q_{J} -1) \} 
   \ \text{ and }\  
a(\varphi) := \rank_k(I/I\cap\fp \wJ) \,,
  \]
    \end{chunk}

  \begin{Theorem}
    \label{thm:optimal}
If $d,c,q,a$ are integers that satisfy $d \ge c \ge q,a  \ge 0$, then there exist a c.i.\ local ring $(R,\fm,k)$ 
and a residue ring $S$ of $R$ with $\fm^3 S =0$ such that  
    \begin{gather}
      \label{eq:optimal1}
(\edim R,\codim R,\codim \qr R, a(\varphi)) = (d,c,q,a) \quad\text{and}
   \\  
   \label{eq:optimal2}
\gran_R(S) = \max\{c - q - 1\,,0\} .
    \end{gather}

As a consequence, the upper bound in Theorem \emph{\ref{thm:gran}} is optimal.
  \end{Theorem}

The proof of the theorem, presented in \ref{ch:pfthmoptimal}, has two crucial ingredients.
The first one is a closed formula for the Poincar\'e series of Golod residue rings of c.i.\ rings. 

  \begin{chunk}
     \label{ch:relations}
We assign nicknames to invariants of $R$, $S$, and $\varphi$ defined in \ref{ch:regular}, \ref{ch:Betti}, and \ref{ch:varphi}:
  \begin{equation}
    \label{eq:relations1}
   \begin{alignedat}{4}
d &:=  \edim R \,, \quad  & c & :=  \codim R \,, \quad & q  &:=  \codim \qr R \,,  \quad  & r & := \rel R \,;  
  \\
e & :=  \edim S \,, & m &:= m(S) \,; & a & := a(\varphi) \,.
   \end{alignedat}
  \end{equation}
These numbers compare as follows:
   \begin{equation}
    \label{eq:relations2}
0 \le q \le c \le d \ge e \ge m \ge 0 \le a \le r \ge c \ \text{ and }\ r = c \iff R \ \text{ is c.i.} 
  \end{equation}

Recall that the ring $S$ is said to be \emph{Golod} if it satisfies the relation
  \[
P^S_k = (1+z)^e/(1 - z^2 P^Q_{J})
  \]
for some (and hence, for every -- see \ref{ch:Betti}) minimal Cohen presentation $\wh S \cong Q/J$.

If $R$ is c.i.\ and $S$ is Golod, then the following equality holds; see \cite{Av:golod}:
   \begin{align}
    \label{eq:relations3}
P^R_{S} &= \frac{(1 + z)^{a + 1}(1 - z)^a + z^2P^Q_{J} -1}{z(1 +z )^{c - d + e}(1 - z)^{c}} \,.
   \end{align}
The very special case $S = R/\fm^2$ of this result first appeared as \cite[Theorem 2.1]{Av:rational}.
   \end{chunk}
   
The second ingredient is the next theorem, where we identify families of Golod residue rings $S$ of an 
\emph{arbitrary} c.i.\ ring $R$ and express their granularities in terms of the numbers in \eqref{eq:relations1}.  
The construction of the rings $S$ and the computation of their invariants utilize a different set of 
techniques; they are deferred to Section~\ref{S:A family of Golod homomorphisms}.

 \begin{Theorem}
     \label{thm:gring}
Let $(Q,\fq,k)$ be a regular local ring of dimension $e\ge1$, $\bsu$ a regular system of parameters,
$U$\! a $2\times(h+1)$ matrix with $h\ge1$ and entries in $\bsu\cup\{0\}$, $I_2(U)$ the ideal 
generated by the $2 \times 2$ minors of $U$, and $J := I_2(U) + \fq^3$.  Put $S := Q/J$, let
$\varphi \colon R \tra S$ be a surjective ring map, and let $d,c,a,m$ be as in \eqref{eq:relations1}.

If $U$ is adequate for $\bsu$ (see \emph{\ref{ch:det}}), then $h \le e$ holds and $S$ is Golod.  
If, furthermore, $R$ is c.i., then $\gran_R(S)$ depends on the position of $a$ respective to $h$ and $e$, as follows:
    \begin{enumerate}[\quad\rm(a)]
    \item
$h = e$; this is equivalent to $S = Q/\fq^2$ and it implies $m = e$ and 
 \begin{equation}
  \label{eq:gring1}
\gran_R(S) = 
  \begin{cases}
\max\{ c - d + e - a - 1 , 0\} &\text{if}\quad a \le e - 2\,;
 \\
0 &\text{if}\quad a \ge e-1 \,.
     \end{cases}
         \end{equation}
    \item
$h \le e - 1$; this implies $m = h + 1$ and 
 \begin{equation}
   \label{eq:gring2}
\gran_R(S) = 
  \begin{cases}
\max\{c - d + e - a - 1 , 0\}  &\text{if}\quad a \le h - 1 \,;
  \\
\max\{ c - d + e - h - 1 , 0\}  &\text{if}\quad a \ge h \,. 
  \end{cases}
           \end{equation}
   \end{enumerate}
  \end{Theorem}

  \begin{proof} 
The case $s = 2$ of Theorem \ref{thm:generic}(2) shows that $S$ is Golod, $h \le e$ holds, and $h = e$ 
is equivalent to $J = \fq^2$; in addition, it yields an equality
   \[
z^2P^Q_{J} - 1 =  \frac{(1+z)^{e}}{z} \left(1-ez + \frac{(e+h+1)(e-h)}{2}z^2\right) + \frac{(1+z)^{h+1}}{z}(hz - 1) \,.
   \]

To compute $\gran_R(S)$ we feed the above expression into Formula \eqref{eq:relations1}, write $P^R_S$ 
as a rational function, evaluate the order of its pole at $-1$, and refer to \ref{ch:rational}(ii).

(a)  When $h = e$ holds, we get $z^2P^Q_{J} - 1 = (1+z)^e(ez - 1)$, whence $m=e$ and  
   \[
P^R_S = \frac{(1 + z)^{a + 1}(1 - z)^a + (1+z)^{e} (ez - 1)}{z(1+z)^{c-d+e}(1-z)^c} \,.
  \]
If $a \le e - 2$ the highest power of $(1+z)$ that divides the numerator is $a+1$; this verifies the 
order of the pole of $P^R_S$ at $-1$ announced in \eqref{eq:gring1}.  If $a\ge e-1$, then that 
highest power is $e + 1$ when $(a, e) = (2,3)$, and $e$ otherwise; in neither case does $P^R_S$ 
have a pole at $-1$, as $c - d \le 0$ holds.  Now the proof of \eqref{eq:gring1} is complete.

(b) When $h \le e-1$ holds, we have $z^2P^Q_{J} - 1 = (1+z)^{h+1} \cdot p(z)$ with 
  \[
p(z) :=  \frac{(1+z)^{e-h-1}}{z} \left(1-ez + \frac{(e+h+1)(e-h)}{2}z^2\right) + \frac{1}{z}(hz - 1) \,.
   \]
The equalities $p(-1) = h+1$ if $h \le e-2$ and $p(-1) = -(h+2)$ if $h = e - 1$ show that $m = h + 1$
holds in both cases. Therefore \eqref{eq:relations3} takes the form
 \[
P^R_S = 
  \left\{
\begin{alignedat}{3}
\frac{(1-z)^a + (1+z)^{h-a} \cdot p(z)}{z(1+z)^{c-d+e-a-1}(1-z)^c} 
 &\quad\text{if } a  \le h - 1 \,;
    \\
\frac{(1+z)^{a-h}(1-z)^a + p(z)}{z(1+z)^{c-d+e-h-1}(1-z)^c} 
 &\quad\text{if } a \ge h \,.
 \end{alignedat}
 \right.
  \]
Evaluating the numerators of $P^R_S$ at $z = -1$ yields $2^a$ if $a \le h-1$ and $p(-1) \ne 0$ if $a > h$; therefore 
\eqref{eq:gring2} holds when $a \ne h$. When $a = h$ the formula above becomes
  \[
P^R_S=\frac{(1 - z)^h + p(z)}{z(1+z)^{c-d+e-h-1}(1-z)^c}  \,.
  \]
At $z = -1$ the numerator equals $2^h + h + 1$ if $h \le e-2$ and $2^h - h - 2$ if $h = e-1$; this settles
\eqref{eq:gring2} except if $(h, e)= (2,3)$, where $(1 - z)^2 + p(z) = z(z +1)$ yields
  \[
\gran_R(S) = \max\{c - d - 1 , 0\} = 0 = \max\{c-d+e-h-1, 0\} \,.
  \qedhere
   \]
  \end{proof}

Appropriate choices, in that order, of a matrix $U$ and of a ring $R$ in Theorem~\ref{thm:gring}  
provide the setup for proving that the upper bound in \eqref{eq:gran1} is optimal.

  \begin{chunk}
    \label{ch:pfthmoptimal}
  \emph{Proof of Theorem \emph{\ref{thm:optimal}}.}
Let $(P,\fp,k)$ be a $d$-dimensional regular local ring, $e$ an integer satisfying $0 \le e \le d$, 
and $\{t_1,\dots,t_{d-e}\}\sqcup\{u_{1},\dots,u_{e}\}$ a regular system of parameters for~$P$.  
Thus $Q := P/(t_1,\dots,t_{d-e})$ is regular and the canonical map $P \tra Q$ sends 
$\{u_1,\dots,u_e\}$ bijectively onto a minimal set of generators of $\fq := \fp Q$.

If $q > a$ the hypothesis yields $d \ge c \ge q > a \ge 0$, and hence the number $e := d - q + a$ 
satisfies $d - e = q - a > 0$; define residue rings of $P$ by setting 
  \begin{alignat*}{4}
R &:= P / I  &\quad&\text{with}\quad &I &:= &(t_1^2,\dots,t_{q-a}^2) & \, + \, (u_{1}^2,\dots,u_a^2) + (u_{q+1}^4,\dots,u^4_c) \,;
    \\
S &:= P / \wJ &\quad&\text{with}\quad &\wJ &:= &(t_1,\dots,t_{q-a}) & \, + \,(u_1,\dots,u_q)^2 + (u_{1},\dots,u_e)^3 \,.
  \end{alignat*}
If $q \le a$ holds, then the line-up is $d \ge c \ge a \ge q \ge 0$.  Choose $e = d$ and set
  \begin{alignat*}{3}
R &:= P / I  &\quad&\text{with}\quad &I &:= (u_1^2,\dots,u_q^2) +(u_{q+1}^3,\dots,u_{a}^3) +(u_{a+1}^4,\dots,u^4_c) \,;
    \\
S &:= P / \wJ  &\quad&\text{with}\quad & \wJ &:= (u_1,\dots,u_q)^2 + (u_{1},\dots,u_d)^3 \,.
  \end{alignat*}

It is clear that $R$ is c.i\ with $(\edim R,\codim R,\codim \qr R) = (d,c,q)$.  Choose $\varphi \colon R \tra S$ 
to be the homomorphism defined by $I \subseteq \wJ$; in both cases it is easy to see that $\{u_1,\dots,u_a\}$ is 
a $k$-basis of $I/I\cap\fp \wJ$, and this yields $a(\varphi_*) = a$.  The ideal
  \[
\wJ = I_2(U) + \fq^3 \quad\text{with}\quad
U := \begin{bmatrix} u_1 & u_2 & \dots & u_q & 0 \\ 0 & u_1 & \dots & u_{q-1} & u_q\end{bmatrix}
  \]
of $Q$ satisfies $\wJ := J Q$ and $Q/\wJ = S$.  Applying Theorem \ref{thm:gring} with $h = q$ yields
  \[ 
\gran_R(S) = 
\left\{ 
\begin{alignedat}{2}
\max\{ c - d + (d - q + a) - a  - 1, 0\} &= \max\{ c - q  - 1, 0\}  &\text{ if } a & \le q - 1 \,;
    \\
\max\{ c - d + d - q  - 1, 0\} &= \max\{ c - q  - 1, 0\} &\text{ if } a & \ge q \,. 
  \end{alignedat} 
\right.
  \] 

The proof of Theorem \ref{thm:optimal} is complete.  \qed
  \end{chunk}

As another application of Theorem \ref{thm:gring}, we show that the existence of residue rings 
$S$ with $\fn^2 =0$ and $\gran_R(S) = 0$ imposes upper bounds on $\codim R$.
 
\begin{Proposition}
      \label{prop:loewy} 
Let $(R,\fm,k)$ be a c.i.\ local ring, $\wh R \cong P/I$ a minimal Cohen presentation, 
$\fp$ the maximal ideal of $P$, and $L$ a proper ideal of $P$ satisfying $L^*_2 \supseteq \qri I2$.

If $S :=  P/(L + \fp^2)$ has $\gran_R(S) = 0$, then the following inequalities hold:
  \begin{align}
    \label{eq:loewy1}
\codim R - 1 &\le \rank_k L^*_1 +  \rank_k \big(\qri I2/(\qri I2\cap\gri P1L^*_1)\big) \,.
  \\
    \label{eq:loewy2}
\codim R - 1 &\le \rel \qr R \quad\text{if}\quad L \subseteq \fp^2 \,.
  \end{align}
      \end{Proposition}

  \begin{proof} 
Choose a subset $\bst$ of $L$ that is mapped bijectively onto some $k$-basis of $L/L\cap\fp^2$.
The ideal $\wJ := P\bst + \fp^2$ satisfies $\wJ = L + \fp^2$, $\wJ^*_1= L^*_1$, and 
  \[
\frac{\qri I2}{\qri I2\cap\gr P_1J^*_1}
\cong \frac{\qri I2+\gr P_1J^*_1}{\gr P_1J^*_1} 
\cong \frac{(I+\fp(P\bst+\fp^2))/\fp^3}{\fp(P\bst+\fp^2)/\fp^3} 
\cong \frac{I+\fp J}{\fp J} 
\cong \frac{I}{I\cap \fp J} \,.
   \]
With notation from \ref{ch:relations}, we get $\rank_k(\qri I2/\qri I2\cap\gr P_1J^*_1) = a$, and hence  
the right-hand side of \eqref{eq:loewy1} equals $d - e + a$.  As $\fq^2 S = 0$, Theorem \ref{thm:gring}(a) 
applies to $R \tra S$, and \eqref{eq:gring1} yields $c-1\le d - e + a$ when $a \le e - 2$.  On the other 
hand, when $a \ge e - 1$ we get $c - 1 \le d - 1 \le d - e + a$ from the relation $c \le d$; see 
\eqref{eq:relations2}.  Now Formula \eqref{eq:loewy1} has been proved.  Formula \eqref{eq:loewy2}  
records the special case $L^*_1 = 0$.
   \end{proof} 

The proposition yields a stronger and sharper version of \cite[Theorem B]{Av:rational}.

  \begin{Corollary}
    \label{cor:quadrics}
If $(R,\fm,k)$ is a local ring with $\cx_R(\fm^2) < \infty$ and $\gran_R(\fm^2) =0$, then $R$ is c.i.~and 
$\codim R \le \rel \qr R + 1$ holds.
  \end{Corollary}

  \begin{proof}
Put $L := R/\fm^2$.  From \ref{ch:basics}(1) and \cite[Theorem 4 and Proposition 2]{Av:extremal}, one gets 
$\cx_R(k) = \cx_R(L) = \cx_R(\fm^2) < \infty$; thus $R$ is c.i.\ and $\gran_R(L) = \gran_R(\fm^2) = 0$ holds; 
see \ref{ch:ci}(3) and \ref{ch:basics}(1). Now Formula \eqref{eq:loewy2} yields $\codim R \le \rel \qr R + 1$.
  \end{proof}

 \section{Eventually polynomial Betti sequences}
\label{S:Eventually polynomial Betti sequences}

Recall that $(R,\fm,k)$ denotes a local ring and $M$ a finite $R$-module.  

We say that the Betti sequence $(\beta^R_i(M))$ is \emph{eventually polynomial} if there exists 
$\beta^{R,M}\in\QQ[z]$ such that $\beta^R_i(M)=\beta^{R,M}(i)$ holds for $i\gg0$.  In this section 
we look for conditions on the structure of $R$ that imply or follow from the property that the Betti 
sequence of \emph{every} finite $R$-module is eventually polynomial.  

The next result yields the homogeneous case in Theorem \ref{thm:low-I} in the introduction; 
it answers, in the positive, a question raised at the end of the introduction of~\cite{Av:rational}.

  \begin{Theorem}
    \label{thm:strict}
If $A$ is a standard graded $k$-algebra and $R$ is its localization at the maximal ideal $(A_1)$, 
then the following conditions are equivalent.
  \begin{enumerate}[\quad\rm(i)]
 \item
The Betti sequence of $R/\fm^2$ is eventually polynomial. 
  \item
The Betti sequence of each finite $R$-module is eventually polynomial. 
 \item
The ring $R$ is c.i.\ and satisfies $\codim R \le \codim \qr R + 1$.
  \item 
The graded algebra $\gr R$ is c.i.\ and has at most one non-quadratic relation.
  \end{enumerate}
   \end{Theorem}

  \begin{proof}
Let $\pi \colon \Sym_k(A_1) \tra A$ be the canonical map of graded $k$-algebras.  Localizing 
$\pi$ at the maximal ideal $(A_1)$ yields a minimal regular presentation $R \cong P/I$ with 
$I = \Ker(\pi)P$ and isomorphisms $\gr R \cong A \cong \gr P/I^*$ of graded \text{$k$-algebras};
cf.~\ref{ch:sym}.  They induces isomorphisms $\qr R \cong \gr P/I^*_2$ of graded algebras and  
$I/\fp I \cong I^*/ \gri P1I^*$ of  $k$-vector spaces, where $\fp$ is the maximal ideal of $P$.  

(i)$\implies$(iv).
In view of Corollary \ref{cor:quadrics}, $R$ is c.i.\ with $\codim R - \rel \qr R \le 1$.  Choose 
$f_1,\dots, f_r \in I$ such that $\{f_1^*,\dots, f_r^*\}$ minimally generates $I^*$ and 
$\{f_1^*,\dots, f_b^*\}$ is a $k$-basis of $I^*_2$.  From $I/\fp I \cong I^*/ \gri P1I^*$ we get 
$r = \codim R$ and $I = (f_1,\dots, f_r)$, and hence $\{f_1,\dots, f_r\}$ is $P$-regular; therefore 
$\{f^*_1,\dots, f^*_r\}$ is $\gr P$-regular; see \ref{ch:VV}(1).  From $\qr R \cong \gr P/I^*_2$ we 
get $\rel \qr R = \rank_k I^*_2 = b$, whence $r - b \le 1$, as desired.

(iv)$\implies$(iii).
Choose $f_1,\dots, f_c \in I$ such that $\{f^*_1,\dots, f^*_c\}$ is $\gr P$-regular, generates $I^*$, and 
$\deg(f^*_j) = 2$ for $1 \le j \le c-1$.  As $R$ is isomorphic to the localization of $\gr R$ at $(P^*)$ 
(due to $A \cong \gr P/I^*$), the image of $\{f^*_1,\dots, f^*_c\}$ in $P$ is a regular set that generates 
$I$, and $\{f^*_1,\dots, f^*_{c-1} \}$ is $k$-independent in $(I_2 + (P_1)^2) / (P_1)^3 \cong I^*_2$.

(iii)$\implies$(ii).
This implication follows from Theorem \ref{thm:gran}.

(ii)$\implies$(i).
This implication is a tautology.
   \end{proof}

The part of Theorem \ref{thm:low-I} concerning c.i.\ rings of low codimension comes from

 \begin{Theorem}
     \label{thm:low}
When $(R,\fm,k)$ is a local ring whose cyclic modules $S$ with $\fm^j S=0$ for some  
$j \ge 1$ have eventually polynomial Betti sequences, then $R$ is c.i., and
  \begin{equation}
   \label{eq:low1}
\codim R - \codim\qr R \le\max\{\codim R - i, 1\} 
  \end{equation} 
holds for $i \le 2$ if $j=2$, and also for $i = 3$ if $j=3$ and $k$ is algebraically~closed.
 \end{Theorem}

This theorem is proved in \ref{ch:pf4.2}.  The argument draws upon a classical description of the 
homogeneous prime ideals of codimension two and minimal multiplicity in polynomial rings over 
algebraically closed fields; Huneke, Mantero, McCullough, and Seceleanu \cite{HMMS1} used it
to bound the projective dimensions of those ideals.  We review the relevant parts, using notation 
that will facilitate the references in \ref{ch:pf4.2}.  

 \begin{chunk}
     \label{ch:delpezzo}
Let $k$ be an algebraically closed field, $\gr P$ a polynomial ring over $k$ with variables $\{u^*_1,\dots,u^*_d\}$ 
of degree one, and $D$ a homogeneous prime ideal of $\gr P$.  The ideal $D$ is said to be \emph{degenerate} 
if $D_1 \ne0$, and \emph{non-degenerate} otherwise; in the latter case, a well known inequality involves the 
multiplicity of $\gr P/D$
(see \cite[Proposition 0]{EH}):
  \begin{equation}
    \label{eq:delpezzo1}
e(\gr P/D)\ge\height D+1 \,.
   \end{equation}

Homogeneous prime ideals of height two admit explicit descriptions, possibly after some change of variables. The 
ideal $D$ is degenerate if and only if $D=(u^*_1,g^*_2)$ with $g^*_2$ an irreducible form in $k[u^*_2,\dots,u^*_d]$; 
in this case, $e(\gr P/D)=\deg(g^*_2)$.  

Non-degenerate $D$ belong to one of two types.  If $e(\gr P/D) = 3$, then $D$ is the ideal 
generated by the $2\times2$ minors of one of the matrices $U^*$, displayed below:
 \begin{equation}
    \label{eq:delpezzo2}
\begin{bmatrix} u^*_1 & u^*_2 & u^*_3 \\ u^*_4 & u^*_1 & u^*_2\end{bmatrix},
\ \ \text{or}\ \ 
\begin{bmatrix} u^*_1 & u^*_2 & u^*_3 \\ u^*_4 & u^*_5 & u^*_2\end{bmatrix},
\ \ \text{or}\ \ 
\begin{bmatrix} u^*_1 & u^*_2 & u^*_3 \\ u^*_4 & u^*_5 & u^*_6\end{bmatrix}.
   \end{equation}
If $e(\gr P/D) \ne 3$, then $D=(g^*_1,g^*_2)$ for some $\gr P$-regular set $\{g^*_1,g^*_2\}$ of forms and 
$e(\gr P/D)=\deg(g^*_1)\deg(g^*_2)$.  This classification was obtained in \cite{X} and \cite[Theorem 3]{SD} 
(see also \cite[Theorem 1]{EH}); it is described as above in \cite[p.~63]{En}.
   \end{chunk}
  
  \begin{chunk}\emph{Proof of Theorem \emph{\ref{thm:low}}.}
    \label{ch:pf4.2}
The Betti sequence of $k$ is eventually polynomial, by assumption, and therefore $R$ is c.i.; see \ref{ch:ci}(3).  
Replacing $R$ with $\wh R$ does not change the hypothesis of the theorem (as both rings have the same 
modules of prescribed Loewy length), nor its conclusion (see \eqref{eq:grR3}).  Thus we may assume 
$R = P/I$ with $(P,\fp,k)$ a regular local and $I$ generated by a regular sequence in $\fp^2$. 

Put $c := \codim R$ and $q := \codim \qr R$, and hence $q = \height \qr I$; see \eqref{eq:grR1}.  
We show by (a very short!) induction on $i$ that if certain cyclic $R$-modules have eventually polynomial 
Betti sequences, then $c \ge i +1$ implies $q \ge i$ for $0 \le i \le 3$. 

There is nothing to prove when $i = 0$.  Suppose that the cyclic $R$-modules $S$ with $\fm^2 S = 0$ have 
$\gran_R(S) = 0$.  If $i = 1$, then from Corollary \ref{cor:quadrics} we get $\rel \qr R \ge c-1 \ge 1$, whence 
$\qr I \ne 0$, and hence $q \ge 1$.  When $i = 2$, the claim is that $\qr I \ne 0$ and $c \ge 3$ implies $q \ge 2$.  
Indeed, $q = 1$ means that $\qr I$ is contained in $\gr Pg^*$ for some $g \in \fp$ with $\deg(g^*) = 1$ or 
$\deg(g^*) = 2$, and then Formula \eqref{eq:loewy1} applied with $L := Pg$ yields $2 \le c-1\le 1$, which is absurd.  

Now assume that $k$ is algebraically closed and the cyclic $R$-modules $S$ annihilated by $\fm^3$ have 
$\gran_R(S) = 0$.  We claim that $\qr I \ne 0$ and $c \ge 4$ implies $q \ge 3$.  For the sake of contradiction, 
suppose $q = 2$; then $\gr P$ has prime ideals of height $2$ that contain $I$, and they are homogeneous;  let 
$D$ be one of them.  As $k$ is infinite, there exist $f_1,f_2\in I$ such that $\{\qri f1,\qri f2\}$ is $\gr P$-regular.  The 
surjective ring homomorphisms $\gr P/(\qri f1,\qri f2)\tra \gr P/\qr I\tra \gr P/D$ yield inequalities of multiplicities, to wit
  \[
4=e(\gr P/(\qri f1,\qri f2))\ge e(\gr P/\qr I)\ge e(\gr P/D)\ge1\,.
  \]
We rule out every admissible value of $e(\gr P/D)$ by using the classification in \ref{ch:delpezzo}.

When $e(\gr P/D) \ne 3$, one has $D = (g^*_1,g^*_2)$ for some regular set $\{g^*_1,g^*_2\}$ 
of forms with $n_i := \deg(g^*_i)$ satisfying $1 \le n_1 \le n_2 \le 2$.  We prove that the 
existence of such a set implies $c \le 3$, which is ruled out by our hypothesis.  Indeed, 
$D$ equals $L^*$ for $L := (g_1,g_2) \subset P$; see \ref{ch:VV}(1).  If $n_2 = 1$, then 
$\gri P1 L^*_1 = L^*_2 \supseteq \qri I2$ holds and \eqref{eq:loewy1} yields $c - 1 \le 2+0$. If 
$n_1 < n_2$, then we have $L^*_1=kg^*_1$ and $L^*_2 = \gr P_1 g^*_1 \oplus kg^*_2$; as $\qri I2$ 
contains a $\gr P$-regular set of two elements, we get $\qri I2 \not\subseteq \gr P_1 g^*_1$, whence 
$\gr P_1 g^*_1 + \qri I2 = L^*_2$, and hence $\qri I2/(\qri I2\cap \gr P_1 L^*_1) \cong kg^*_2$; now 
\eqref{eq:loewy1} gives $c - 1 \le 1 + 1 $.  Finally, $n_1=2$ implies $c - 1 \le 0+2$, again by 
\eqref{eq:loewy1}.

If $e(\gr P/D)=3$, then $D = (y_1,y_2,y_3)$, where the $y_j$s are the $2\times2$ minors of a 
$2\times3$ matrix $U^*$ in Formula \eqref{eq:delpezzo2} and $u^*_i$ is the initial form of $u_i$, 
where $\bsu := \{u_1,\dots,u_d\}$ of $\fp$ is a minimal generating set .  Let $U$ be the matrix 
obtained from $U^*$ by replacing each $u^*_i$ with $u_i$, and let $g_j$ be the minor of $U$ 
that corresponds to $y_j$.  As $g^*_j = y_j$ holds for $j = 1,2,3$, the ideal $L := (g_1,g_2,g_3)$ 
of $P$ has $L^*_1 = 0$ and $L^*_2 = D_2$; also, $I$ lies in $\wJ := L + \fp^3$, as seen from 
the relations
  \[
\frac{I+\fp^3}{\fp^3} = \qri I2 \subseteq D_2 = \qri L2 = \frac{L+\fp^3}{\fp^3} = \frac{\wJ}{\fp^3} \,.
  \]
Put $Q := P$ and $S := P/ \wJ$, and let $d,c,a,e$ be the numbers assigned in \ref{ch:relations} to 
the canonical map $R \tra S$.  The matrix $U$ is adequate for $\bsu$ (cf.~\ref{ch:det}), so Theorem 
\ref{thm:gring} applies.  Here we have $e = d$ and $h = 2$, and therefore the granularity of $S$ is
given by Formula \eqref{eq:gring2}.  When $\gran_R(S) = 0$ this formula yields $c - a - 1 \le 0$ 
if $a \le 1$, and $c - 3 \le 0$ if $a \ge 2$.  We end up with $4 \le c \le 3$ and therefore $q \ge 3$ holds.
  \qed
\end{chunk}

  \section{A family of Golod homomorphisms}
    \label{S:A family of Golod homomorphisms}

This section does not rely on material in earlier parts of the paper.  The goal is Theorem 
\ref{thm:generic}, which contains results crucial to the proofs in Sections \ref{S:The upper 
bound is optimal} and \ref{S:Eventually polynomial Betti sequences}.

 \begin{chunk}
     \label{ch:gradedBetti}
In this section $k$ denotes a field, $\bsx$ a finite set of indeterminates of degree one, and $A$ a $k$ algebra 
isomorphic to $k[\bsx]/I$, where $I$ is a homogeneous ideal in $(\bsx)^2$.  Furthermore, $N$ denotes a graded 
$A$-module; we set $\inf N := \inf\{ j \in \ZZ \mid N_j \ne0 \}$ if $N \ne 0$ and $\inf 0 = \infty$; abusing notation, 
we write $k$ for $A/(A_1)$.  

A blanket hypothesis is that all $A$-modules are graded and finitely generated, 
their submodules are homogeneous, and their homomorphisms preserve degrees.

Natural gradings $\Tor^A_i(N,k)=\bigoplus_{j\in\ZZ}\Tor^A_i(N,k)_j$ are inherited from  
resolutions by free graded $A$-modules.  The \emph{graded Betti numbers}
$\beta^A_{i,j}(N):=\rank_k\Tor^A_i(N,k)_j$ satisfy the conditions
$\beta^A_{i,j}(N)=0$  for $i \notin [0,\pd_AN]$, $j < i + \inf N$, and $j\gg i$.

We write $P^A_N(y,z)$, or $P^A_N$, for the \emph{graded Poincar\'e series} of $N$, defined to be
  \[
P^A_N(y,z) := \sum_{i\ges0}\sum_{j\in\ZZ} \beta^A_{i,j}(N) y^jz^i \in\ZZ[y^{\pm1}][[z]]\,.
  \]

  \begin{enumerate}[\rm(1)]
    \item
Localization at $(A_1)$, denoted here by $? \leadsto ?^{\ell}$, is a faithfully exact functor from 
graded $A$-modules to $A^{\ell}$-modules; it preserves freeness and minimality, whence
  \[
P^{A^{\ell}}_{N^{\ell}}(z) = P^A_N(1,z) \,.
  \]
   \end{enumerate}

Three relevant properties of graded algebras are defined in terms of $P^A_k$.
  \begin{enumerate}[\rm(1)]
    \item[\rm(2)]
The algebra $A$ is said to be \emph{Koszul} if it satisfies the condition
   \[
H_{A}(-yz) \cdot P^A_k(y,z) = 1 \,;
   \]
see also \ref{ch:linear}(1).  In particular, $k[\bsx]$ is Koszul and $P^{k[\bsx]}_k(y,z) = (1+yz)^{|\bsx|}$.
    \item[\rm(3)]
The algebra $A$ is said to be a \emph{graded complete intersection} if $I = k[\bsx]\,\bsg$ for some 
$k[\bsx]$-regular set of forms, $\bsg$; by a graded version of \ref{ch:ci}(3), this is equivalent to 
   \[
 P^A_k(y,z) \cdot \prod_{g \in \bsg} (1 - y^{\deg(g)}z^2)  = (1+yz)^{|\bsx|} \,.
   \]
    \item[\rm(4)]
The algebra $A$ is said to be \emph{Golod} if it satisfies the condition
   \[
P^A_k(y,z) \cdot \left(1 - z^2P^{k[\bsx]}_{I}(y,z) \right) = (1+yz)^{|\bsx|}  \,.
   \]
(Note:  This equality differs from that in \cite{HRW}, where Formula (2.2) is incorrect.)
   \end{enumerate}

In view of \ref{ch:gradedBetti}(1), the algebra $A$ is c.i., respectively, Golod if and only if the local ring 
$A^{\ell}$ has the corresponding property; cf.~\ref{ch:ci}, respectively, \ref{ch:relations}. 
   \end{chunk}

The focus in this section is on specific properties of polynomial ideals.

    \begin{chunk}
   \label{ch:linear}
Let $B := k[\bsx]$ be a polynomial ring, $I$ an ideal of $B$, and $t$ an integer.

We say that $I$ is $t$-\emph{linear} (or, $I$ \emph{has an $t$-linear resolution}) if $\beta^B_{i,j}(I) = 0$ 
holds for $j\ne i+t$; for instance, $(\bsx)$ is $1$-linear.  The ideal $I$ is \emph{linear} if it is $t$-linear 
for some $t$; when $I$ is linear, it is $(\inf I)$-linear if $I \ne 0$, and $t$-linear for each $t \in \ZZ$ if 
$I = 0$.  

Following Herzog and Hibi 
\cite{HH}, we say $I$ is \emph{componentwise linear} if the ideal $\lin Ij := BI_j$ is $j$-linear for each 
$j \in \ZZ$.  We list a few relevant properties.
   \begin{enumerate}[\rm(1)]
   \item
If $I$ is $t$-linear, then the following equality holds: 
   \begin{align*}
(-z)^t P^B_I(y,z) & = (1 + yz)^{|\bsx|} H_I(-yz) \,.
   \end{align*}
   \item
If $I$ is linear, then it is componentwise linear.
   \item
If $I$ is componentwise linear, then the following equality holds:
  \[
\beta^B_{i,j}(I) = \sum_{h \in \ZZ} \beta_{i,h}(\lin I{j} ) - \sum_{h \in \ZZ} \beta_{i,h}(B_1 \lin I{j-1})
\quad\text{for}\quad  i, j\in \ZZ \,.
  \]
    \end{enumerate}

Part (1) can be read off Polishchuk and Positselski \cite[Proof of Proposition~2.2]{PP}.  
(2) is well known; e.g., \cite[1.1]{PP} or \cite[Lemma 1]{HRW}.  Part (3) is proved in \cite[1.3]{HH}.  
     \end{chunk}

We reduce computation of Poincar\'e series of componentwise linear ideals to a 
potentially simpler task -- computing Hilbert series of finitely many residue rings.

 \begin{Proposition}
     \label{prop:component}
Let $B := k[\bsx]$ be a polynomial ring, $I$ an ideal of $B$, and put
\begin{equation}
 \label{eq:component1}
e := |\bsx|  \,, \ 
n_j := \rank_k I_j \,, \  
\JJ := \{ j \in \ZZ \mid I_{j} \ne B_1 I_{j - 1} \} \,, \ 
\text{and} \, \ 
t := \max \JJ \,.
  \end{equation}

If $I$ is componentwise linear, then $P^B_{I}(z,y)$ is given by the following formula:
 \begin{equation}
 \begin{aligned}
  \label{eq:component2}
P^B_I (y, z)
& = (1 + yz)^{e} \sum_{j \in \JJ}  (-z)^{-j} \big(H_{B/\lin I{j-1}}(-yz) - H_{B/\lin I{j}} (-yz) \big)
  \\
&  {}\phantom{ = \ }
 + (1 + yz)^{e} \sum_{j \in \JJ}  (-z)^{-j} n_{j-1} (-yz)^{j-1} \,.
  \end{aligned}
  \end{equation}
   \end{Proposition} 

   \begin{proof}
Both $\lin I{j}$ and $B_1 \lin I{j-1}$ are $j$-linear, by the assumption on $I$ and by 
\ref{ch:linear}(2), respectively; thus, the formula in \ref{ch:linear}(3) is shorthand for
a family of equalities:
 \[
\beta^B_{i,j}(I) = 
  \begin{cases}
\beta^B_{i,i+j} (\lin I{j}) - \beta^B_{i,i+j}(B_1 \lin I{j-1}) & \text{for } j \in \JJ \,;
  \\
0 & \text{otherwise} \,. 
  \end{cases}
  \]

Multiply the $i$th equality by $y^{i+j}z^{i}$ and sum up over $i \in \ZZ$; the result is  
  \[
\sum_{i \in \ZZ} \beta^B_{i,j}(I) y^{i+j} z^i = 
  \begin{cases}
P^B_{\lin I{j}} - P^B_{B_1\lin I{j-1}} &\text{for } j \in \JJ \,;
   \\
0 & \text{otherwise} \,. 
  \end{cases}
 \]
Now multiply each power series by $(-z)^{t}$, and aggregate the products:
  \begin{align*}
(-z)^{t} P^B_I 
= (-z)^{t} \sum_{i,j} \beta^B_{i,j}(I) y^{i+j} z^i 
= \sum_{j \in \JJ} (-z)^{t-j} \big((-z)^{j} P^B_{\lin I{j}} - (-z)^{j} P^B_{B_1\lin I{j-1}} \big) \,.
  \end{align*}

Multiply the last equality by $H_B(-yz)$ and invoke \ref{ch:linear}(1) to get 
  \begin{align*}
(-z)^{t} H_B(-yz) P^B_I 
& = \sum_{j \in \JJ} (-z)^{t-j} \big((-z)^{j} H_B(-yz) P^B_{\lin I{j}} - (-z)^{j} H_B(-yz) P^B_{B_1\lin I{j-1}} \big)
  \\	
& = \sum_{j \in \JJ} (-z)^{t-j} (H_{\lin I{j}} (-yz) - H_{B_1\lin I{j-1}} (-yz) ) 
  \\	
& = \sum_{j \in \JJ} (-z)^{t-j} \big(H_{\lin I{j}} (-yz) - (H_{\lin I{j-1}}(-yz) - n_{j-1} (-yz)^{j-1})\big)
  \\	
& = \sum_{j \in \JJ} (-z)^{t-j} (H_{B/\lin I{j-1}}(-yz) - H_{B/\lin I{j}} (-yz) + n_{j-1} (-yz)^{j-1}) 
  \end{align*}
The preceding equalities, multiplied by $(-z)^{-t} (1+yz)^e$, yield  \eqref{eq:component2}; see \ref{ch:gradedBetti}(2).
  \end{proof}

  \begin{chunk}
    \label{ch:det}
Let $Q$ be a noetherian ring and $L$ a nonzero $Q$-module.  The maximal length of $Q$-regular sequences 
in $\Ann_Q(L)$ is called the \emph{grade} of $L$; it is denoted by $\grade L$ and satisfies 
$\grade L \le \pd_Q L$. If equality holds (and $\grade L = g$), then $L$ is said to be \emph{perfect} (of 
grade $g$).  When $Q$ is regular, $\grade L = \height \Ann_Q(L)$ holds, and $L$ is perfect of grade $g$ if and 
only if $L$ is Cohen-Macaulay and $\dim L = \dim B - g$.

Let $\bsx := \{x_1,\dots,x_e\}$ be a set of nonzero elements of $Q$. Let $X = \begin{bmatrix} x_{i,j} \end{bmatrix}$ 
be an $s\times (s + h - 1)$ matrix with entries from $\bsx \cup \{0\}$, and put
  \[
\bsx' := \{x_1,\dots,x_h\}\,, \ 
\Delta^X_n := \{x_{i,j}\}_{j-i+1=n} \text{ for } n \in [1,h]\,,  \text{ and } \, \Delta^X := \bigcup_{n=1}^h \Delta^X_n .
  \]  

We say that $X$ is \emph{adequate for} $\bsx$ if the following conditions are satisfied: 
  \[
\Delta^X \subseteq \bsx \,; \quad
x_l \in \Delta^X_n \iff  l = n \in [1,h] \,; \quad 
\big{|} \{ \, (i,j)  :  x_{i,j} = x_n \in \bsx \setminus \bsx' \big\} \big{|} \le 1  \,.
  \]
  
Let $I(X,Q)$ denote the ideal of $Q$ generated by the $s \times s$ minors of $X$.
   \end{chunk}

Eagon and Northcott \cite[Theorem 2]{EN1} proved that $I(X)$ is perfect of grade $h$ if the entries of 
$X$ are distinct indeterminates.  In the results that follow we describe families of ideals with similar 
properties that are parametrized by fewer variables.  

  \begin{Lemma}
      \label{lem:adequate}
Let $k$ be a field, $\bsx := \{x_1,\dots,x_e\}$ a set of indeterminates, and put $B := k[\bsx]$.  Let $X$ 
be an $s\times (s + h - 1)$ matrix that is adequate for $\bsx$ (see \emph{\ref{ch:det}}).

The module $A(X) := B/I(X)$ is perfect of grade $h$, the ideal $I(X)$ is $s$-linear,~and 
 \begin{equation}
   \label{eq:adequate1} 
(-z)^{s} P^B_{I(X)} (-yz) =  1 - (1 + yz)^{h} \sum_{i=0}^{s-1} \binom{h-1+i}{i} (-yz)^i  \,.
 \end{equation}

In the special case $h = e$ one has $I(X) = (\bsx)^s$.
  \end{Lemma}

    \begin{proof}
Put $A(X) := k[\bsx]/I(X)$. For each $n \in [1,h]$ choose $i \in [1,s]$ such that $x_{i,i+n-1} = x_n$.  
Let $\sigma \colon k[\bsx] \to k[\bsx]$ be the $k$-algebra map that swaps  $x_{1,i + n -1}$ and 
$x_{1,n}$ for $i \in [1,h]$ and fixes the rest of $\bsx$.  As $\sigma$ induces an isomorphism of 
$k$-algebras $A(X) \cong A(\sigma(X))$, we may suppose that $x_{1,n} = x_n$ holds for $n \in [1,h]$.

Let $Y = \begin{bmatrix} y_{i,j} \end{bmatrix}$ be an $s\times (s + h - 1)$ matrix with distinct entries 
from a set $\bsy$ of $s\times (s + h - 1)$ indeterminates such that $\bsy \cap \bsx =\emptyset$; put
$\bsy' := \{y_{1,1},\dots, y_{1,h}\}$ and $C := k[\bsy']$.  The module $A(Y)$ is perfect of grade $h$,  
the following is $A(Y)$-regular
  \[
\bsz_Y  := \{ y_{i,i+n-1} - y_{1,n} \}_{i \ne 1, n \in [1,h]} \cup \{ y_{i,j} \notin \Delta_Y \} \,,
  \]
and $A(Y)/ \bsz_Y A(Y) \cong C/(C_1)^s$ holds; see Eagon \cite[Proof of Theorem~1]{Ea}.

Let $\varkappa \colon k[\bsy] \to k[\bsx] = B$ be the $k$-algebra map with $y_{i,j} \mapsto x_{i,j}$.  Since 
$\bsx$ is adequate for $X$, the ideal $\Ker(\varkappa)$ is generated by the following set of linear forms:
 \begin{equation*}
\bsz_X :=  
\{\, y_{i,i+n-1} - y_{1,n} \mid x_{i,i+n-1} = x_{n}\, \}_{i \ne 1, n \in [1,h]} 
 \cup \{\, y_{i,j} \mid x_{i,j} = 0\, \}  
  \end{equation*}
The set $\bsz_X$ is $A(Y)$-regular (as it is a part of $\bsz_Y$) and $A(X) \cong A(Y)/\bsz_X A(Y)$ holds; 
thus $A(X)$ is perfect of grade $h$.  In addition,  the set of linear forms
$\bsz := \varkappa(\bsz_Y \setminus \bsz_X)$ is $A(X)$-regular with $A(X)/ \bsz A(X) \cong C/(C_1)^s$.
Therefore $A(X)$ is perfect of grade $h$ and $P^B_{A(X)} = P^{C}_{C/(C_1)^s}$ holds.  This implies 
$P^B_{I(X)} = P^{C}_{(C_1)^s}$; in particular, $I(X)$ is $B$-linear, by \ref{ch:linear}(2).  Setting 
$\bsy'' := \bsy \setminus \bsy'$ and applying \eqref{eq:component2} with $\JJ = \{ s \}$ yields 
 \[
(-z)^{-s} P^B_I (y, z) = (1 + yz)^{e} \big(H_{k[\bsy]}(-yz) - H_{k[\bsy'']}(-yz) H_{C/(C_1)^s} (-yz) \big) \,.
  \]
It remains to plug in the well known expressions of the Hilbert series involved.
     \end{proof}

  \begin{chunk}
      \label{ch:generic}
Let $\bsx = \{x_1,\dots,e_e\}$ be a set of indeterminates, $D := \ZZ[\bsx]$, $I$ a homogeneous ideal, and 
$g := \grade D/I$.  When $Q$ is a noetherian ring and $\bsu=\{u_1,\dots,u_e\} \subset Q$, put $I(\bsu, Q) := Q \phi(I)$, 
where $\phi \colon D \to Q$ is the ring map with $\phi(x_i) = u_i$ for~$i \in [1,e]$.  
    \begin{enumerate}[\rm(1)]
   \item
When $D/I$ is $\ZZ$-free the following conditions are equivalent: (i) $D/I$ is perfect of grade $g$; 
(ii) $K[\bsx] / I(\bsx, K[\bsx])$ is perfect of grade $g$ for every finite prime field~$K$; 
(iii) $K[\bsx] / I(\bsx, K[\bsx])$ is perfect of grade $g$ for every noetherian ring $K$.
   \end{enumerate}

The ideal $I$ called \emph{generically perfect} of grade~$g$ if it satisfies the conditions in~(1).
In case it does and $I(u,Q) \ne Q$ holds, $Q/I(u,Q)$ has the following properties:
   \begin{enumerate}[\rm(1)]
   \item[\rm(2)]
$\grade Q/I(u,Q) \le g$, and $Q/I$ is perfect if equality holds.  
   \item[\rm(3)]
If $(Q,\fq,k)$ is a local ring, $\bsu \subset \fq$, and $\grade Q/I(\bsu,Q) \ge g$ holds, then one has
   \[
P^Q_{Q/I(\bsu,Q)}(z) = P^{B^{\ell}}_{B^{\ell}/I(\bsx, B)^{\ell}}(z)  = P^{B}_{B/I(\bsx, B)}(1,z)\quad\text{with}\quad B := k[\bsx] \,.
   \]
   \end{enumerate}

See Hochster \cite[Theorem 1]{Ho}, complemented by \cite[Proposition~20]{HE} for Part~(1);
\cite[Proposition~4]{EN2} for Part (2); \cite[Proof of Theorem~6.2]{Av:small} and \ref{ch:gradedBetti}(1) for Part (3). 
  \end{chunk}

  \begin{Corollary}
      \label{cor:generic}
Let $\bsx$ be a set of indeterminates, $X$ an $s\times (s + h - 1)$ matrix that is adequate for $\bsx$
(see \emph{\eqref{ch:det}}), and put $D := \ZZ[\bsx]$.
   \begin{enumerate}[\rm(1)]
    \item
The ideal $I := I(X, D)$ is generically perfect of grade $h$ (see \emph{\ref{ch:generic}}).
    \item
The ideal $J := I(X, D) + (\bsx)^{t+1}$ is generically perfect of grade $e$.
   \end{enumerate}
   \end{Corollary}
 
     \begin{proof}
(1)  One has $k \otimes_{\ZZ} I = I(X, k[\bsx])$ for every field $k$.  From Lemma 
\ref{lem:adequate} we know that $I(X, k[\bsx])$ is perfect of grade $g$ and that
that $H_{k[\bsx]/I(X, k[\bsx])}(y)$ does not depend on $k$; thus for each $j \in \ZZ$ and every 
prime number $p$ the $\ZZ/p\ZZ$-rank of $(D/I)_j / p (D/I)_j$ equals the $\QQ$-rank of 
$(D/I)_j \otimes_{\ZZ} \QQ$.  Therefore $D/I$ is $\ZZ$-free.

(2)  As $(D/J)_{j} = (D/I)_{j}$ for $j \le s$ and $(D/J)_{j} = 0$ for $j > s$, Part (1) shows that
$D/J$ is $\ZZ$-free.  It is clear that $k[\bsx]/(I(X, k[\bsx]) + (\bsx)^{s+1})$ is perfect of grade~$e$.
     \end{proof}

 \begin{chunk}
      \label{ch:golod}
Recall that a ring homorphism $Q \tra Q/J$ is said to be \emph{Golod} if $J \subseteq \fq^2$ and 
$P^{Q/J}_k(z) = P^Q_k(z)/(1 -z^2 P^Q_J)$ holds; see Levin \cite{Le}.  Thus a local ring $S$ is Golod if 
and only a minimal Cohen presentation $Q \tra \wh S$ is a Golod homomorphism;~cf.~\ref{ch:relations}.
  \end{chunk}

  \begin{Theorem}
      \label{thm:generic}
Let $(Q,\fq,k)$ be a local ring, $\bsu=\{u_1,\dots,u_e\}$ a $Q$-regular subset of $Q$, and $U$ 
an $s\times (s + h - 1)$ matrix with entries in $\{\bsu\}\cup\{0\}$ that is adequate for~$\bsu$; 
see \emph{\ref{ch:det}}.  Let $I(U)$ be the ideal of $Q$ generated by the $s \times s$ minors of $U$.
   \begin{enumerate}[\rm(1)]
     \item
The ideal $I := I(U)$ is perfect of grade $h$, and $P^Q_I(z) = P^B_{I(X)}(1,z)$; see \eqref{eq:adequate1}.
    \item
The ideal $J := I(U)+(\bsu)^{s+1}$ is perfect of grade $e$, and 
  \begin{equation*}
\begin{aligned}
z^2P^Q_{J} 
& = \frac{1} {(-z)^{s-2}} + \frac{(1+z)^{h+1}}{(-z)^{s-1}}\left(\sum_{i=0}^{s-1}\binom{h-1+i}{i}(-z)^{i}\right) 
  \\
& \phantom{\ = } - \frac{(1+z)^{e}}{(-z)^{s-1}}\left(\sum_{i=0}^{s}\binom{e-1+i}{i}(-z)^{i} - \binom{h-1+s}{s}(-z)^s\right) .
    \end{aligned}
   \end{equation*}
    \item
When $s \ge 2$ the canonical maps $Q/I \tla Q \tra Q/J$ are Golod homomorphisms.
   \end{enumerate}
   \end{Theorem}
 
    \begin{proof}
(1)  Let $S$ denote the associated graded ring of the ideal $(\bsu)$ and $a^*$ the initial form of $a \in Q$.  Note 
that $\bsu^* := \{u_1^*, \dots , u_e^*\}$ is algebraically independent over the subring $K := Q/(u_1,\dots,u_e)$; 
we identify $S$ and $K[\bsx]$, write $X$ for the matrix $U^* = \begin{bmatrix} u_{i,j}^* \end{bmatrix}$, and do 
not distinguish between $I(U^*,S)$ and $I(X,K[\bsx]$), see~\ref{ch:det}.  

It is clear that $X$ is adequate for $\bsx$; then $I(\bsx,\ZZ[\bsx])$ is generically perfect, by Corollary 
\ref{cor:generic}(1), and hence $I(X,K[\bsx])$ is perfect of grade $h$, by \ref{ch:generic}(1).  Therefore so is 
$I(U)$ (see Northcott, \cite[Proposition~3]{No}),  and \ref{ch:generic}(3) yields $P^Q_I(z) = P^B_{I(X)}(1,z)$. 

(2)  The ideal $J$ is perfect of grade $e$, by Corollary \ref{cor:generic}(2) and \ref{ch:generic}(2), so we 
have $P^Q_{J} = P^{k[\bsx]}_{J(\bsx, k[\bsx])}(1,z)$, by \ref{ch:generic}(3). The ideal $J(\bsx, k[\bsx])$ is 
componentwise linear, with $\lin {J(\bsx,k[\bsx])}{s} = I(X)$ and $\lin {J(\bsx,k[\bsx])}{s+1} = (\bsx)^{s+1}$.
Applying Proposition \ref{prop:component} with $\JJ = \{s,s+1\}$ and substituting the expressions for 
$P^{k[\bsx]}_{k[\bsx/I(\bsx)}$ and $P^{k[\bsx]}_{k[\bsx]/(\bsx)^{s+1}}(y,z)$, from Lemma \ref{lem:adequate},
into Formula \eqref{eq:component2} yields $P^{k[\bsx]}_{J(\bsx, k[\bsx])}(y,z)$.  Now refer to \ref{ch:gradedBetti}(1). 
 
(3)  The map $Q \tra Q/J$ is Golod if and only if the ring $(k[\bsx]/J(\bsx,k[\bsx]))^{\ell}$ is Golod  
(see \cite[Theorem~6.2]{Av:small}), if and only if the algebra $k[\bsx]/J(\bsx,k[\bsx])$ is Golod  
(see \ref{ch:golod}); this algebra is Golod because the ideal $J(\bsx,k[\bsx])$ is componentwise
 linear; therefore Herzog, Reiner, and Welker, \cite[Theorem 4]{HRW} applies. 
    \end{proof}

\end{document}